\documentclass[a4paper,11pt]{amsart}

\usepackage{amssymb}
\usepackage{mathrsfs}
\usepackage{cite}
\usepackage{url}

\newtheorem{theorem}{Theorem}[section]
\newtheorem{lemma}[theorem]{Lemma}
\newtheorem{corollary}[theorem]{Corollary}
\newtheorem{proposition}[theorem]{Proposition}

\theoremstyle{definition}

\numberwithin{equation}{section}

\makeatletter
\@namedef{subjclassname@2020}{%
  \textup{2020} Mathematics Subject Classification}
\makeatother

\DeclareMathOperator{\RE}{Re}

\DeclareMathOperator{\rad}{rad}
\DeclareMathOperator{\norm}{N}

\begin{document}

\title[Upper bound for the number of smooth values of a polynomial]
{An upper bound for the number of smooth values of a polynomial and its applications}

\author[M. Mine]{Masahiro Mine}
\address{Global Education Center\\ Waseda University\\
1-6-1 Nishiwaseda, Shinjuku-ku, Tokyo 169-8050, Japan}
\email{m-mine@aoni.waseda.jp}

\date{}

\begin{abstract}
We prove a new upper bound for the number of smooth values of a polynomial with integer coefficients. 
This improves Timofeev's previous result unless the polynomial is a product of linear polynomials with integer coefficients. 
As an application, we provide another proof for a result of Cassels which was used to prove that the Hurwitz zeta-function with algebraic irrational parameter has infinitely many zeros on the domain of convergence. 
We also apply the main result to a problem on primitive divisors of quadratic polynomials. 
\end{abstract}



\maketitle

\section{Introduction}\label{sec:1}

\subsection{Statement of the main results}\label{sec:1.1}
Let $P^+(m)$ denote the largest prime factor of an integer $m$ with conventions $P^+(\pm1)=1$ and $P^+(0)=+\infty$. 
An integer $m$ is said to be \textit{$y$-smooth} or \textit{$y$-friable} if $P^+(m) \leq y$ holds. 
The study of the distribution of smooth numbers is a classical topic in analytic number theory. 
Denote by $\Psi(x,y)$ the number of integers $n$ in the interval $[1,x]$ such that $n$ is $y$-smooth. 
In 1930, Dickman \cite{Dickman1930} showed that 
\begin{align*}
\Psi(x,x^{1/u})
\sim \rho(u) x
\end{align*}
as $x \to\infty$ for every fixed $u>0$, where $\rho$ is the continuous solution of the system
\begin{align*}
\begin{cases}
\rho(u)=1
& \text{for $0 \leq u \leq 1$, }
\\
u\rho'(u)+\rho(u-1)=0
& \text{for $u>1$.}
\end{cases}
\end{align*}
The function $\rho$ is known today as the Dickman function. 
See the surveys \cite{Granville2008, HildebrandTenenbaum1993} for more results on the asymptotic behavior of $\Psi(x,y)$. 

Let $f(t) \in \mathbb{Z}[t]$. 
Denote by $\Psi_f(x,y)$ the number of integers $n$ in $[1,x]$ such that $f(n)$ is $y$-smooth. 
Suppose that $f(t)$ has distinct irreducible factors over $\mathbb{Z}[t]$ of degrees $d_1, \ldots, d_g \geq1$, respectively. 
It was conjectured by Martin \cite{Martin2002} that 
\begin{align}\label{eq:09251118}
\Psi_f(x,x^{1/u})
\sim \rho(d_1u) \cdots \rho(d_gu) x
\end{align}
as $x \to\infty$ for every fixed $u>0$. 
No result is known that \eqref{eq:09251118} holds unconditionally for any non-linear polynomial $f(t)$ unless $u$ is too small to make it trivial. 
Then upper and lower bounds for $\Psi_f(x,y)$ have been studied. 
Hmyrova \cite{Hmyrova1966} proved that, if $f(t)$ is irreducible, then 
\begin{align*}
\Psi_f(x,x^{1/u})
< c(f) \exp\left(-u \log{\frac{u}{e}}\right) x
\end{align*}
holds in the range $e \leq u \leq (\log{x})(\log\log{x})^{-1}$ for sufficiently large $x$, where $c(f)$ is a positive constant depending on the degree and coefficients of $f(t)$. 
See also a recent article of Suzuki \cite{Suzuki2025+} for an improvement of Hmyrova's estimate. 
For more general polynomials, Timofeev \cite{Timofeev1977} obtained that
\begin{align*}
\Psi_f(x,x^{1/u})
< \frac{(g+\epsilon)^{[u]}}{d(d-1)^{[u]-1}u^{[u]}} x
\end{align*}
holds in the range $1 \leq u \leq \sqrt{(\log{x})(d+\epsilon)^{-1}}$ for sufficiently large $x$ and every $\epsilon>0$, where we put $d=d_1+\cdots+d_g$ under the same assumption as in \eqref{eq:09251118}. 
Here, and throughout this paper, $[u]$ denotes the largest integer not exceeding $u$. 
We omit the known results on lower bounds for $\Psi_f(x,y)$ because the theme of this paper is about the upper bounds. 
The statement of the main result is as follows. 

\begin{theorem}\label{thm:1.1}
Suppose that $f(t) \in \mathbb{Z}[t]$ has distinct irreducible factors over $\mathbb{Z}[t]$ of degrees $d_1, \ldots, d_g \geq1$, respectively. 
If $d=d_1+\cdots+d_g \geq2$, then 
\begin{align}\label{eq:09242338}
\Psi_f(x,x^{1/u})
< \gamma_f(u)
\frac{g^{[u]}}{d(d-1)^{[u]-1}u^{[u]}}
\left(x+O\left(\frac{x}{\log\log{x}}\right)\right)
\end{align}
holds in the range $1 \leq u \leq \sqrt{\log{x}}\, (\log\log{x})^{-1}$ for sufficiently large $x$, where $\gamma_f(u)$ is a positive real number which can be explicitly represented as
\begin{align}\label{eq:10031455}
\gamma_f(u)
= \frac{1}{2}+\frac{2g+1}{16du}+\sqrt{\frac{2g+1}{16du}+\left(\frac{2g+1}{16du}\right)^2}. 
\end{align}
The implied constant in \eqref{eq:09242338} depends only on $f(t)$.
\end{theorem}

Suppose further that at least one of the irreducible factors of $f(t)$ is not a linear polynomial. 
Then we see that $g<d$ is satisfied, which implies 
\begin{align*}
\frac{2g+1}{16du}
\leq \frac{2d-1}{16d}
< \frac{1}{8}
\end{align*}
for any $u \geq1$. 
In this case, we obtain
\begin{align*}
\gamma_f(u)
< \frac{1}{2}+\frac{1}{8}+\sqrt{\frac{1}{8}+\frac{1}{64}}
= 1. 
\end{align*}
Hence, \eqref{eq:09242338} improves on Timofeev's upper bound unless $f(t)$ is a product of linear polynomials. 
Furthermore, if $f(t)$ is an irreducible polynomial of degree $d \geq2$, then 
\begin{align}\label{eq:10031502}
\gamma_f(u)
= \frac{1}{2}+\frac{3}{16du}+\sqrt{\frac{3}{16du}+\left(\frac{3}{16du}\right)^2} 
\leq \frac{19+\sqrt{105}}{32}
= 0.913...
\end{align}
holds for any $u \geq1$. 
While this may seem like only a slight improvement over the previous result, we will see in Sections \ref{sec:1.2} and \ref{sec:1.3} that it has valuable applications to several issues. 

Theorem \ref{thm:1.1} is proved by induction on $m=[u]$. 
In Section \ref{sec:2}, we overview the strategy of the proof and show some estimates in the case $m=1$. 
Then we explain the inductive process in Section \ref{sec:3}. 
The proof of Theorem \ref{thm:1.1} is finally completed in Section \ref{sec:4}.

\subsection{Note to ``Footnote to a note of Davenport and Heilbronn''}\label{sec:1.2}
Let $\alpha$ be an arbitrary algebraic irrational number, and $K=\mathbb{Q}(\alpha)$ the algebraic field generated by $\alpha$ over the rationals. 
Let $\mathfrak{a}$ be the ideal denominator of the fractional ideal $(\alpha)$ of $K$. 
Then we see that $(n+\alpha)\mathfrak{a}$ is an integral ideal for every $n \in \mathbb{Z}$. 
Denote by $C_\alpha(x)$ the number of integers $n$ in $[1,x]$ such that there exists a prime ideal $\mathfrak{p}$ of $K$ with the conditions 
\begin{align}\label{eq:10022256}
\mathfrak{p} \mid (n+\alpha)\mathfrak{a}
\quad\text{and}\quad
\mathfrak{p} \nmid \prod_{\substack{1 \leq m \leq x \\ m \neq n}} (m+\alpha)\mathfrak{a}. 
\end{align}
A lower bound for $C_\alpha(x)$ was used to study the zeros of the Hurwitz zeta-function $\zeta(s,\alpha)$ for $\RE(s)>1$, where $\alpha$ is an algebraic irrational number with $0<\alpha<1$. 
Indeed, Cassels \cite{Cassels1961} applied the estimate $C_\alpha(x)>0.51x$ for large $x$ to showed that $\zeta(s,\alpha)$ has infinitely many zeros in the strip $1<\RE(s)<1+\delta$ for any $\delta>0$, which extends the classical result of Davenport and Heilbronn \cite{DavenportHeilbronn1936a}. 
The relation between $C_\alpha(x)$ and the number of smooth values of a polynomial $f(t) \in \mathbb{Z}[t]$ was firstly mentioned by Worley \cite{Worley1970}, and a more detailed research was presented by Andersson in his unpublished work \cite{Andersson2016+}. 
In Section \ref{sec:5}, we prove the following result. 

\begin{theorem}\label{thm:1.2}
Take an irreducible polynomial $f(t) \in \mathbb{Z}[t]$ with $f(-\alpha)=0$. 
Then the asymptotic formula
\begin{align*}
C_\alpha(x)
= x-\Psi_f(x,x)
+ O\left(\frac{x}{\log{x}}\right) 
\end{align*}
holds for sufficiently large $x$, where the implied constant depends only on $K=\mathbb{Q}(\alpha)$. 
\end{theorem}

Therefore, to show a lower (\textit{resp.}\ upper) bound for $C_\alpha(x)$ is nothing other than to show an lower (\textit{resp.}\ upper) bound for $\Psi_f(x,x)$. 
In particular, Theorem \ref{thm:1.1} implies the following result. 

\begin{corollary}\label{cor:1.3}
Let $\alpha$ be an arbitrary algebraic irrational number. 
Then we have 
\begin{align*}
C_\alpha(x)
\geq \left\{1-\frac{1}{2d}-\frac{3}{16d^2}-\frac{1}{d}\sqrt{\frac{3}{16d}+\left(\frac{3}{16d}\right)^2}\right\} x
\end{align*}
for sufficiently large $x$, where $d=\deg(\alpha) \geq2$.  
In particular, the inequality
\begin{align*}
C_\alpha(x)
> 0.543x 
\end{align*}
holds for sufficiently large $x$. 
\end{corollary}

Corollary \ref{cor:1.3} was essentially obtained by Cassels \cite{Cassels1961} and Worley \cite{Worley1970}. 
Note that the original proofs in \cite{Cassels1961, Worley1970} should be modified due to an erroneous classification of prime ideals. 
See \cite[Remark 4.10]{Mine2022+} for the details. 
The main difference between the methods in \cite{Cassels1961, Worley1970} and this paper is that the proof of Theorem \ref{thm:1.1} does not require any ideal-theoretic arguments.

\subsection{Primitive divisors of $n^2+b$}\label{sec:1.3}
Another application of Theorem \ref{thm:1.1} is about the primitive divisors of quadratic polynomials. 
Let $(A_n)$ be a sequence of integers. 
We say that an integer $d>1$ is a \textit{primitive divisor} of $A_n$ if $d \mid A_n$ and $(d,A_m)=1$ for all nonzero terms $A_m$ with $m<n$. 
Then we define\
\footnote{
In the previous papers, $R_b(x)$ was denoted by $\rho_b(x)$, and usually $\rho_1(x)$ was abbreviated to $\rho(x)$. 
We change the symbols so as to keep that $\rho$ represents the Dickman function. 
}
\begin{align*}
R_b(x)
= \# \left\{n \in \mathbb{Z} \cap [1,x] ~\middle|~ \textrm{$n^2+b$ has a primitive divisor} \right\}
\end{align*}
for $b \in \mathbb{Z}$. 
Suppose that $-b$ is not an integer square. 
Then Everest and Harman \cite{EverestHarman2008} conjectured that $R_b(x) \sim (\log{2})x$ as $x \to\infty$. 
They also proved 
\begin{align*}
0.5324x
< R_b(x)
< 0.905x
\end{align*}
for sufficiently large $x$. 
Recently, Harman \cite{Harman2024} sharpened the result for $R_1(x)$ to
\begin{align*}
0.5377x
< R_1(x)
< 0.86x. 
\end{align*}
Li \cite{Li2024+} also announced an even better upper bound $R_1(x)<0.8437x$. 
In Section \ref{sec:6}, we connect $R_b(x)$ with the number of smooth values of the quadratic polynomial $f(t)=t^2+b$. 

\begin{theorem}\label{thm:1.4}
Let $f(t)=t^2+b$ with $-b$ not an integer square.  
Then we have
\begin{align*}
R_b(x)
= x-\Psi_f(x,x)
+ O\left(\frac{x \log\log{x}}{\log{x}}\right)
\end{align*}
for sufficiently large $x$, where the implied constant depends only on $|b|$. 
\end{theorem}

Applying Theorem \ref{thm:1.1}, we prove a new lower bound for $R_b(x)$. 

\begin{corollary}\label{cor:1.5}
Suppose that $-b$ is not an integer square. 
Then the inequality
\begin{align*}
R_b(x)
> 0.543x
\end{align*}
holds for sufficiently large $x$. 
\end{corollary}

For $b=1$, Corollary \ref{cor:1.5} is also related to a problem on arctangent relations which was originally studied by Todd \cite{Todd1949}. 
For $n \in \mathbb{Z}_{\geq1}$, we say that $\arctan n$ is \textit{reducible} if there exist integers $f_1, \ldots, f_k$ and $n_1,\ldots,n_k$ with $1 \leq n_j<n$ such that 
\begin{align*}
\arctan n
= f_1 \arctan n_1+\cdots+f_k \arctan n_k
\end{align*}
holds. 
For example, $\arctan 3$ is reducible by $\arctan 3= 3\arctan 1-\arctan 2$. 
On the contrary, we say that $\arctan n$ is \textit{irreducible} if it is not reducible. 
Then we define 
\begin{align*}
N(x)
= \# \left\{n \in \mathbb{Z} \cap [1,x] ~\middle|~ \text{$\arctan n$ is irreducible} \right\}. 
\end{align*}
Chowla and Todd \cite{ChowlaTodd1949} conjectured that $N(x) \sim (\log{2})x$ as $x \to\infty$. 
Kowalski \cite{Kowalski2004} applied an equidistribution result on the roots of $X^2+1 \equiv 0 ~(\bmod\,{p})$ to show that
\begin{align*}
N(x)
> (2-\epsilon+o(1)) \frac{x}{\log{x}}
\end{align*}
as $x \to\infty$ for any $\epsilon>0$. 
Actually, it is not difficult to see that $N(x)=R_1(x)$ holds; see Lemmas \ref{lem:6.1} and \ref{lem:6.2}. 
Hence we obtain the following result. 

\begin{corollary}\label{cor:1.6}
For sufficiently large $x$, we have $N(x)> 0.543x$. 
\end{corollary}

It would be worth noting that $N(x)<0.8437x$ also follows from Li's work \cite{Li2024+}. 
If \eqref{eq:09251118} is true for $f(t)=t^2+b$ and $u=1$, then Theorem \ref{thm:1.4} yields that
\begin{align*}
R_b(x)
\sim (1-\rho(2))x
\end{align*}
as $x \to\infty$. 
Recall that $\rho(u)=1-\log{u}$ for $1 \leq u \leq 2$. 
Hence \eqref{eq:09251118} contains the conjectures of Everest--Harman and Chowla--Todd on $R_b(x)$ and $N(x)$, respectively.  

\subsection*{Acknowledgments}
The author is especially grateful to Yuta Suzuki for sending a copy and translation of Timofeev's paper \cite{Timofeev1977}. 
Many discussions with him were very fruitful. 
The author would also like to thank Hiroshi Mikawa. 
Corollaries \ref{cor:1.5} and \ref{cor:1.6} were motivated by his lecture at Hachiouji Number Theory Seminar 2024. 
The work of this paper was supported by JSPS Grant-in-Aid for Early-Career Scientists (Grant Number 24K16906).

\section{Strategy and the initial estimates}\label{sec:2}
We begin by a simple observation. 
Let $f(t) \in \mathbb{Z}[t]$ be any polynomial of positive degree, and $x,y,z \in \mathbb{R}$ with $x>z \geq1$ and $y \geq1$. 
Suppose that the inequality 
\begin{align}\label{eq:09272331}
\Psi_f(x,y)-\Psi_f(z,y)
< V+ \sqrt{\Psi_f(x,y)-\Psi_f(z,y)} \sqrt{W}
\end{align}
holds with some parameters $V, W>0$. 
Then it can be easily seen that
\begin{align}\label{eq:10022027}
\Psi_f(x,y)-\Psi_f(z,y)
< V+\frac{W}{2}+\sqrt{VW+\frac{W^2}{4}}
\end{align}
is satisfied. 
Therefore we would like to find suitable $V$ and $W$ such that \eqref{eq:09272331} holds. 
Here, we notice that we may suppose that $f(t) \to\infty$ as $t \to\infty$ without loss of generality since $\Psi_{-f}(x,y)=\Psi_f(x,y)$ holds by definition. 
In this case, there exists a constant $T_0(f) \geq2$ such that 
\begin{align}\label{eq:09291328}
f(t_1)
> f(t_2)
> 1
\end{align}
is valid for every $t_1,t_2 \in \mathbb{R}$ with $t_1>t_2>T_0(f)$. 
The following result is the initial step in the investigation of $V$ and $W$. 

\begin{proposition}\label{prop:2.1}
Let $f(t) \in \mathbb{Z}[t]$ be a polynomial such that $f(t) \to\infty$ as $t \to\infty$, and $x,y,z \in \mathbb{R}$ with $x>z>T_0(f)$ and $y \geq1$. 
Then inequality \eqref{eq:09272331} holds with
\begin{align*}
V
&= \frac{1}{\log{f(z)}} \sum_{\substack{1 \leq k \leq f(x) \\ \sqrt{y} < P^+(k) \leq y}} \Lambda(k)
\sum_{\substack{z<n \leq x \\ k \mid f(n) \\ P^+(f(n)) \leq y}} 1, \\
W
&= \frac{1}{(\log{f(z)})^2} \sum_{\substack{1 \leq k_1,k_2 \leq f(x) \\ P^+(k_1) \leq \sqrt{y} \\ P^+(k_2) \leq \sqrt{y}}} \Lambda(k_1) \Lambda(k_2)
\sum_{\substack{z<n \leq x \\ [k_1,k_2] \mid f(n) \\ P^+(f(n)) \leq y}} 1, 
\end{align*}
where $[k_1,k_2]$ denotes the least common multiple of $k_1$ and $k_2$. 
\end{proposition}

\begin{proof}
The idea of the proof is borrowed from Cassels \cite{Cassels1961}. 
It is deduced from \eqref{eq:09291328} that $\log{f(n)}>\log{f(z)}>0$ for $n>z>T_0(f)$. 
Thus we have
\begin{align*}
\Psi_f(x,y)-\Psi_f(z,y)
= \sum_{\substack{z<n \leq x \\ P^+(f(n)) \leq y}} 1 
< \frac{1}{\log{f(z)}} \sum_{\substack{z<n \leq x \\ P^+(f(n)) \leq y}} \log{f(n)}
\end{align*}
by the definition of $\Psi_f(x,y)$. 
Considering the prime factorization of $f(n)$, we obtain
\begin{align*}
\sum_{\substack{z<n \leq x \\ P^+(f(n)) \leq y}} \log{f(n)}
= \sum_{\substack{z<n \leq x \\ P^+(f(n)) \leq y}} 
\sum_{\substack{p, v \\ p^v \mid f(n)}} \log{p}, 
\end{align*}
where $\sum_{p, v}$ stands for the double sum over prime numbers $p$ and integers $v \geq1$. 
Let $\Lambda(n)$ denote as usual the von Mangoldt function. 
Then we have
\begin{align*}
\sum_{\substack{z<n \leq x \\ P^+(f(n)) \leq y}} 
\sum_{\substack{p, v \\ p^v \mid f(n)}} \log{p}
&= \sum_{\substack{z<n \leq x \\ P^+(f(n)) \leq y}} 
\sum_{\substack{1 \leq k \leq f(x) \\ k \mid f(n) \\ P^+(k) \leq y}} \Lambda(k) \\
&= \sum_{\substack{z<n \leq x \\ P^+(f(n)) \leq y}} 
\sum_{\substack{1 \leq k \leq f(x) \\ k \mid f(n) \\ \sqrt{y} < P^+(k) \leq y}} \Lambda(k)
+ \sum_{\substack{z<n \leq x \\ P^+(f(n)) \leq y}} 
\sum_{\substack{1 \leq k \leq f(x) \\ k \mid f(n) \\ P^+(k) \leq \sqrt{y}}} \Lambda(k). 
\end{align*}
From the above, we see that the inequality
\begin{align}\label{eq:09281553}
\Psi_f(x,y)-\Psi_f(z,y)
&< \frac{1}{\log{f(z)}} 
\sum_{\substack{z<n \leq x \\ P^+(f(n)) \leq y}} 
\sum_{\substack{1 \leq k \leq f(x) \\ k \mid f(n) \\ \sqrt{y} < P^+(k) \leq y}} \Lambda(k) \\
&\quad
+ \frac{1}{\log{f(z)}}
\sum_{\substack{z<n \leq x \\ P^+(f(n)) \leq y}} 
\sum_{\substack{1 \leq k \leq f(x) \\ k \mid f(n) \\ P^+(k) \leq \sqrt{y}}} \Lambda(k) \nonumber 
\end{align}
holds. 
The first sum of the right-hand side of \eqref{eq:09281553} is evaluated as 
\begin{align}\label{eq:09281558}
\sum_{\substack{z<n \leq x \\ P^+(f(n)) \leq y}} 
\sum_{\substack{1 \leq k \leq f(x) \\ k \mid f(n) \\ \sqrt{y} < P^+(k) \leq y}} \Lambda(k)
= \sum_{\substack{1 \leq k \leq f(x) \\ \sqrt{y} < P^+(k) \leq y}} \Lambda(k)
\sum_{\substack{z<n \leq x \\ k \mid f(n) \\ P^+(f(n)) \leq y}} 1
\end{align}
by changing the order of the summations. 
As for the second sum of the right-hand side of \eqref{eq:09281553}, we apply the Cauchy--Schwarz inequality to deduce
\begin{align*}
\sum_{\substack{z<n \leq x \\ P^+(f(n)) \leq y}} 
\sum_{\substack{1 \leq k \leq f(x) \\ k \mid f(n) \\ P^+(k) \leq \sqrt{y}}} \Lambda(k)
\leq \sqrt{\Psi_f(x,y)-\Psi_f(z,y)}
\sqrt{\sum_{\substack{z<n \leq x \\ P^+(f(n)) \leq y}} 
\Bigg(\sum_{\substack{1 \leq k \leq f(x) \\ k \mid f(n) \\ P^+(k) \leq \sqrt{y}}} \Lambda(k)\Bigg)^2}. 
\end{align*}
Furthermore, we have the identity
\begin{align*}
\sum_{\substack{z<n \leq x \\ P^+(f(n)) \leq y}} 
\Bigg(\sum_{\substack{1 \leq k \leq f(x) \\ k \mid f(n) \\ P^+(k) \leq \sqrt{y}}} \Lambda(k)\Bigg)^2
= \sum_{\substack{1 \leq k_1,k_2 \leq f(x) \\ P^+(k_1) \leq \sqrt{y} \\ P^+(k_2) \leq \sqrt{y}}} \Lambda(k_1) \Lambda(k_2)
\sum_{\substack{z<n \leq x \\ [k_1,k_2] \mid f(n) \\ P^+(f(n)) \leq y}} 1
\end{align*}
by expanding the square and changing the order of summations. 
Hence we arrive at
\begin{align}\label{eq:09281559}
&\sum_{\substack{z<n \leq x \\ P^+(f(n)) \leq y}} 
\sum_{\substack{1 \leq k \leq f(x) \\ k \mid f(n) \\ P^+(k) \leq \sqrt{y}}} \Lambda(k) \\
&\leq \sqrt{\Psi_f(x,y)-\Psi_f(z,y)}
\sqrt{\sum_{\substack{1 \leq k_1,k_2 \leq f(x) \\ P^+(k_1) \leq \sqrt{y} \\ P^+(k_2) \leq \sqrt{y}}} \Lambda(k_1) \Lambda(k_2)
\sum_{\substack{z<n \leq x \\ [k_1,k_2] \mid f(n) \\ P^+(f(n)) \leq y}} 1}. \nonumber 
\end{align}
Then the desired result follows by inserting \eqref{eq:09281558} and \eqref{eq:09281559} to inequality \eqref{eq:09281553}. 
\end{proof}

\section{An inductive process}\label{sec:3}
By Proposition \ref{prop:2.1}, the next task to deduce an upper bound of $\Psi_f(x,y)-\Psi_f(z,y)$ is evaluating the sum
\begin{align*}
S(\kappa)
= \sum_{\substack{z<n \leq x \\ \kappa \mid f(n) \\ P^+(f(n)) \leq y}} 1
\end{align*}
with $\kappa=k$ and $\kappa=[k_1,k_2]$. 
For every $\kappa \in \mathbb{Z}_{\geq1}$, we define $\omega_f(\kappa)$ as the number of the zeros of the polynomial $f(t)$ in the residue class ring $\mathbb{Z}/\kappa \mathbb{Z}$, that is, 
\begin{align}\label{eq:10101628}
\omega_f(\kappa)
= \# \left\{ u ~(\bmod\,{\kappa}) ~\middle|~ f(u) \equiv 0 ~(\bmod\,{\kappa}) \right\}. 
\end{align}
Recall that $f(j \kappa+u) \equiv f(u) ~(\bmod\,{\kappa})$ for any $j \in \mathbb{Z}$. 
Therefore we obtain 
\begin{align}\label{eq:09300001}
S(\kappa)
\leq \sum_{\substack{z<n \leq x \\ \kappa \mid f(n)}} 1
= \sum_{\substack{u \hspace{-2mm} \pmod{\kappa} \\ f(u) \equiv 0 \hspace{-2mm} \pmod{\kappa}}} 
\sum_{\substack{z<n \leq x \\ n \equiv u \hspace{-2mm} \pmod{\kappa}}} 1 
\leq \omega_f(\kappa) \left(\frac{x-z}{\kappa}+1\right),
\end{align}
but this estimate is not good when $y$ is small since we have ignored the condition $P^+(f(n)) \leq y$ in the first inequality. 
The purpose of this section is to evaluate $S(\kappa)$ by using $S(\kappa \lambda)$ for $\lambda \geq1$, which enables us to refine inductively the estimates of $V$ and $W$ such that \eqref{eq:09272331} holds. 

\begin{lemma}\label{lem:3.1}
Let $f(t) \in \mathbb{Z}[t]$ be a polynomial such that $f(t) \to\infty$ as $t \to\infty$, and $x,y,z \in \mathbb{R}$ with $x>z>T_0(f)$ and $y \geq1$. 
Suppose that $f(z)>x$ is satisfied, and put $h=x-z$. 
Then we have 
\begin{align}\label{eq:09291603}
\sum_{\substack{z<n \leq x \\ \kappa \mid f(n) \\ P^+(f(n)) \leq y}} 1
&< \frac{1}{\log{\frac{f(z)}{x}}} \sum_{\substack{1 \leq \lambda \leq h/\kappa \\ P^+(\lambda) \leq y}} \Lambda(\lambda)
\sum_{\substack{z<n \leq x \\ \kappa \lambda \mid f(n) \\ P^+(f(n)) \leq y}}  1 \\
&\quad
+ \frac{1}{\log{\frac{f(z)}{x}}} \sum_{\substack{h/\kappa < \lambda \leq f(x) \\ P^+(\lambda) \leq y}} \Lambda(\lambda)
\omega_f(\kappa \lambda) \nonumber
\end{align}
for any $\kappa \in \mathbb{Z}$ with $1 \leq \kappa \leq h$. 
\end{lemma}

\begin{proof}
Since we have $\kappa \leq h<x$ and $f(z)>x$, we find that $\log{\frac{f(n)}{\kappa}}>\log{\frac{f(z)}{x}}>0$ for $n>z>T_0(f)$ by \eqref{eq:09291328}. 
Therefore we obtain the inequality
\begin{align*}
\sum_{\substack{z<n \leq x \\ \kappa \mid f(n) \\ P^+(f(n)) \leq y}} 1
< \frac{1}{\log{\frac{f(z)}{x}}}
\sum_{\substack{z<n \leq x \\ \kappa \mid f(n) \\ P^+(f(n)) \leq y}} \log{\frac{f(n)}{\kappa}}. 
\end{align*}
Considering the prime factorization of $f(n)/\kappa$, we have
\begin{align*}
\sum_{\substack{z<n \leq x \\ \kappa \mid f(n) \\ P^+(f(n)) \leq y}} \log{\frac{f(n)}{\kappa}}
= \sum_{\substack{z<n \leq x \\ \kappa \mid f(n) \\ P^+(f(n)) \leq y}} 
\sum_{\substack{1 \leq \lambda \leq f(x) \\ \lambda \mid f(n)/\kappa \\ P^+(\lambda) \leq y}} \Lambda(\lambda)
\end{align*}
by the same argument as in the proof of Proposition \ref{prop:2.1}. 
Furthermore, we have
\begin{align*}
\sum_{\substack{z<n \leq x \\ \kappa \mid f(n) \\ P^+(f(n)) \leq y}} 
\sum_{\substack{1 \leq \lambda \leq f(x) \\ \lambda \mid f(n)/\kappa \\ P^+(\lambda) \leq y}} \Lambda(\lambda)
&= \sum_{\substack{1 \leq \lambda \leq f(x) \\ P^+(\lambda) \leq y}} \Lambda(\lambda)
\sum_{\substack{z<n \leq x \\ \kappa \lambda \mid f(n) \\ P^+(f(n)) \leq y}}  1 \\
&= \sum_{\substack{1 \leq \lambda \leq h/\kappa \\ P^+(\lambda) \leq y}} \Lambda(\lambda)
\sum_{\substack{z<n \leq x \\ \kappa \lambda \mid f(n) \\ P^+(f(n)) \leq y}}  1 
+ \sum_{\substack{h/\kappa < \lambda \leq f(x) \\ P^+(\lambda) \leq y}} \Lambda(\lambda)
\sum_{\substack{z<n \leq x \\ \kappa \lambda \mid f(n) \\ P^+(f(n)) \leq y}}  1
\end{align*}
by changing the order of summations and dividing the sum over $\lambda$. 
The sum over $\lambda$ with $1 \leq \lambda \leq h/\kappa$ yields the first term of the right-hand side of \eqref{eq:09291603}. 
On the other hand, the sum over remaining $\lambda$ is estimated as
\begin{align*}
\sum_{\substack{h/\kappa < \lambda \leq f(x) \\ P^+(\lambda) \leq y}} \Lambda(\lambda)
\sum_{\substack{z<n \leq x \\ \kappa \lambda \mid f(n) \\ P^+(f(n)) \leq y}}  1
\leq \sum_{\substack{h/\kappa < \lambda \leq f(x) \\ P^+(\lambda) \leq y}} \Lambda(\lambda)
\omega_f(\kappa \lambda)
\end{align*}
since the number of integers $n$ with $z<n \leq x$ and $\kappa \lambda \mid f(n)$ is less than $\omega_f(\kappa \lambda)$ due to $h=x-z< \kappa \lambda$. 
Hence we obtain the second term of the right-hand side of \eqref{eq:09291603}, and the proof is completed. 
\end{proof}

Lemma \ref{lem:3.1} was developed by Timofeev \cite{Timofeev1977}. 
He used this result implicitly in his study on the upper bound for $\Psi_f(x,y)$. 
In this paper, we apply Lemma \ref{lem:3.1} to refine Proposition \ref{prop:2.1} as follows. 

\begin{proposition}\label{prop:3.2}
Let $f(t) \in \mathbb{Z}[t]$ be a polynomial such that $f(t) \to\infty$ as $t \to\infty$, and $x,y,z \in \mathbb{R}$ with $x>z>T_0(f)$ and $y \geq1$. 
Suppose that $f(z)>x$ is satisfied, and put $h=x-z$. 
Then inequality \eqref{eq:09272331} holds with
\begin{align}\label{eq:09291709}
V
= V_m^+ +\sum_{i=1}^{m} V_i^-
\quad\text{and}\quad
W
= W_m^+ +\sum_{i=1}^{m} W_i^-
\end{align}
for every $m \in \mathbb{Z}_{\geq1}$, where $V_m^+$, $W_m^+$, $V_i^-$, $W_i^-$ for $i=1,\ldots,m$ are defined as follows. 
Firstly, we define $V_m^+$ and $W_m^+$ as
\begin{align*}
V_m^+
&= \frac{1}{\log{f(z)} \left(\log{\frac{f(z)}{x}}\right)^{m-1}} 
\sum_{\substack{k_1,\ldots,k_m \geq 1 \\ k_1 \cdots k_m \leq h \\ \sqrt{y} < P^+(k_1) \leq y \\ P^+(k_j) \leq y ~(j \geq 2)}} 
\Lambda(k_1) \cdots \Lambda(k_m) 
\sum_{\substack{z<n \leq x \\ k_1 \cdots k_m \mid f(n) \\ P^+(f(n)) \leq y}} 1, \\
W_m^+
&= \frac{1}{(\log{f(z)})^2 \left(\log{\frac{f(z)}{x}}\right)^{m-1}} \\
&\quad \times
\sum_{\substack{k_1,\ldots,k_{m+1} \geq 1 \\ [k_1,k_2] k_3 \cdots k_{m+1} \leq h \\ P^+(k_j) \leq \sqrt{y} ~(j=1,2) \\ P^+(k_j) \leq y ~(j \geq 3)}} 
\Lambda(k_1) \cdots \Lambda(k_{m+1}) 
\sum_{\substack{z<n \leq x \\ [k_1,k_2] k_3 \cdots k_{m+1} \mid f(n) \\ P^+(f(n)) \leq y}} 1
\end{align*}
with the convention $k_3 \cdots k_{m+1}=1$ for $m=1$. 
Secondly, we define $V_i^-$ as
\begin{align*}
V_i^-
= \frac{1}{\log{f(z)} \left(\log{\frac{f(z)}{x}}\right)^{i-1}} 
\sum_{\substack{k_1,\ldots,k_i \geq 1 \\ k_1 \cdots k_{i-1} \leq h,~ \frac{h}{k_1 \cdots k_{i-1}}<k_i \leq f(x) \\ P^+(k_j) \leq y ~(\forall j)}} 
\Lambda(k_1) \cdots \Lambda(k_i) \omega_f(k_1 \cdots k_i), 
\end{align*}
for every $1 \leq i \leq m$ with the convention $k_1 \cdots k_{i-1}=1$ for $i=1$. 
Lastly, we define 
\begin{align*}
W_i^-
&= \frac{1}{(\log{f(z)})^2 \left(\log{\frac{f(z)}{x}}\right)^{i-1}} \\
&\quad \times
\sum_{\substack{k_1,\ldots,k_{i+1} \geq 1 \\ [k_1,k_2] k_3 \cdots k_{i} \leq h,~ \frac{h}{[k_1,k_2] k_3 \cdots k_{i}}<k_{i+1} \leq f(x) \\ P^+(k_j) \leq y ~(\forall j)}} 
\Lambda(k_1) \cdots \Lambda(k_{i+1}) \omega_f([k_1,k_2] k_3 \cdots k_{i+1}) 
\end{align*}
for every $2 \leq i \leq m$ with the convention $k_3 \cdots k_{i}=1$ for $i=2$, and 
\begin{align*}
W_1^-
= \frac{1}{(\log{f(z)})^2} \sum_{\substack{1 \leq k_1, k_2 \leq f(x) \\ [k_1,k_2]>h \\ P^+(k_j) \leq y ~(\forall j) }} 
\Lambda(k_1) \Lambda(k_2) \omega_f([k_1,k_2]). 
\end{align*}
\end{proposition}

\begin{proof}
The proof is by induction on $m$. 
The result for $m=1$ can be directly deduced from Proposition \ref{prop:2.1}. 
Hence it suffices to show that the inequalities
\begin{align}\label{eq:09291954}
V_m^+ 
< V_{m+1}^+ +V_{m+1}^-
\quad\text{and}\quad
W_m^+ 
< W_{m+1}^+ +W_{m+1}^-
\end{align}
hold for every $m \in \mathbb{Z}_{\geq1}$ since they imply
\begin{align*}
V_m^+ +\sum_{i=1}^{m} V_i^-
< V_{m+1}^+ +\sum_{i=1}^{m+1} V_i^-
\quad\text{and}\quad
W_m^+ +\sum_{i=1}^{m} W_i^-
< W_{m+1}^+ +\sum_{i=1}^{m+1} W_i^-. 
\end{align*}
We apply Lemma \ref{lem:3.1} with $\kappa=k_1 \cdots k_m$ to derive
\begin{align*}
V_m^+
&< \frac{1}{\log{f(z)} \left(\log{\frac{f(z)}{x}}\right)^m} \\
&\quad \times
\sum_{\substack{k_1,\ldots,k_m \geq 1 \\ k_1 \cdots k_m \leq h \\ \sqrt{y} < P^+(k_1) \leq y \\ P^+(k_j) \leq y ~(j \geq 2)}} 
\Lambda(k_1) \cdots \Lambda(k_m) 
\sum_{\substack{1 \leq \lambda \leq \frac{h}{k_1 \cdots k_m} \\ P^+(\lambda) \leq y}} \Lambda(\lambda)
\sum_{\substack{z<n \leq x \\ k_1 \cdots k_m \lambda \mid f(n) \\ P^+(f(n)) \leq y}}  1 \\
&\quad
+ \frac{1}{\log{f(z)} \left(\log{\frac{f(z)}{x}}\right)^m} \\
&\qquad \times
\sum_{\substack{k_1,\ldots,k_m \geq 1 \\ k_1 \cdots k_m \leq h \\ P^+(k_j) \leq y ~(\forall j)}} 
\Lambda(k_1) \cdots \Lambda(k_m) 
\sum_{\substack{\frac{h}{k_1 \cdots k_m} < \lambda \leq f(x) \\ P^+(\lambda) \leq y}} \Lambda(\lambda)
\omega_f(k_1 \cdots k_m \lambda). 
\end{align*}
Replacing $\lambda$ by $k_{m+1}$, we see that the first and second terms of the right-hand side are equal to $V_{m+1}^+$ and $V_{m+1}^-$, respectively. 
Hence we obtain inequality \eqref{eq:09291954} for $V_m^+$. 
Similarly, we apply Lemma \ref{lem:3.1} with $\kappa=[k_1,k_2] k_3 \cdots k_{m+1}$ to derive
\begin{align*}
&W_m^+
< \frac{1}{(\log{f(z)})^2 \left(\log{\frac{f(z)}{x}}\right)^m} \\
&\quad \times
\sum_{\substack{k_1,\ldots,k_{m+1} \geq 1 \\ [k_1,k_2]k_3 \cdots k_{m+1} \leq h \\ P^+(k_j) \leq \sqrt{y} ~(j=1,2) \\ P^+(k_j) \leq y ~(j \geq 3)}} 
\Lambda(k_1) \cdots \Lambda(k_{m+1}) 
\sum_{\substack{1 \leq \lambda \leq \frac{h}{[k_1,k_2] k_3 \cdots k_{m+1}} \\ P^+(\lambda) \leq y}} \Lambda(\lambda)
\sum_{\substack{z<n \leq x \\ [k_1,k_2] k_3 \cdots k_{m+1} \lambda \mid f(n) \\ P^+(f(n)) \leq y}}  1 \\
&\quad
+ \frac{1}{(\log{f(z)})^2 \left(\log{\frac{f(z)}{x}}\right)^m} \\
&\qquad \times
\sum_{\substack{k_1,\ldots,k_{m+1} \geq 1 \\ [k_1,k_2]k_3 \cdots k_{m+1} \leq h \\ P^+(k_j) \leq y ~(\forall j)}} 
\Lambda(k_1) \cdots \Lambda(k_{m+1}) 
\sum_{\substack{\frac{h}{[k_1,k_2] k_3 \cdots k_{m+1}} < \lambda \leq f(x) \\ P^+(\lambda) \leq y}} \Lambda(\lambda)
\omega_f(k_1 \cdots k_m \lambda). 
\end{align*}
In this case, we replace $\lambda$ with $k_{m+2}$. 
Then we see that the first and second terms of the right-hand side are equal to $W_{m+1}^+$ and $W_{m+1}^-$, respectively. 
Then inequality \eqref{eq:09291954} follows also for $W_m^+$.
\end{proof}

\section{Proof of Theorem \ref{thm:1.1}}\label{sec:4}
Before proceeding to the completion of the proof of Theorem \ref{thm:1.1}, we notice that we may suppose that the polynomial $f(t)$ is a product of distinct irreducible factors $f_1(t), \ldots, f_g(t)$ over $\mathbb{Z}[t]$ of degrees $d_1, \ldots, d_g \geq1$, respectively. 
Indeed, if $f(t)$ is factorized in general as 
\begin{align*}
f(t)
= A f_1(t)^{e_1} \cdots f_g(t)^{e_g}
\end{align*}
with $A, e_1,\ldots,e_g \in \mathbb{Z}$ such that $A \neq0$ and $e_j \geq1$, then it is deduced from the definition of $\Psi_f(x,y)$ that $\Psi_{\rad(f)}(x,y) = \Psi_{f}(x,y)$ for $y>|A|$, where we define
\begin{align*}
\rad(f(t))
= f_1(t) \cdots f_g(t). 
\end{align*}
Then we prove the following lemma on sums involving $\omega_f(k)$ defined as \eqref{eq:10101628}. 

\begin{lemma}\label{lem:4.1}
Suppose that $f(t) \in \mathbb{Z}[t]$ is a product of distinct irreducible factors $f_1(t), \ldots, f_g(t)$ over $\mathbb{Z}[t]$ of degrees $d_1, \ldots, d_g \geq1$, respectively. 
Then we have
\begin{align}
\sum_{\substack{1 \leq k \leq x \\ P^+(k) \leq y}} \frac{\Lambda(k)}{k} \omega_f(k)
&\leq g \log{y}+O(1), \label{eq:09292105} \\ 
\sum_{\substack{1 \leq k \leq x \\ \sqrt{y}< P^+(k) \leq y}} \frac{\Lambda(k)}{k} \omega_f(k)
&\leq \frac{g}{2} \log{y}+O(1), \label{eq:09292108} \\ 
\sum_{\substack{1 \leq k \leq x \\ P^+(k) \leq y}} \frac{(2\log{k}-\Lambda(k)) \Lambda(k)}{k} \omega_f(k)
&\leq \frac{g}{2} (\log{y})^2+O(\log{y}), \label{eq:09292106} \\ 
\sum_{\substack{1 \leq k \leq x \\ P^+(k) \leq y}} \Lambda(k) \omega_f(k)
&=O\left(y \frac{\log{x}}{\log{y}}\right) \label{eq:09292107}
\end{align}
for any $x, y \geq2$, where all implied constants depend only on $f(t)$.
\end{lemma}

\begin{proof}
Let $\Delta_f$ denote the absolute value of the discriminant of $f(t)$. 
Note that $\Delta_f$ is a positive integer by the assumption on $f(t)$. 
Then we write the prime factorization of $\Delta_f$ as 
\begin{align*}
\Delta_f
= \prod_{p} p^{\theta(p)} 
\end{align*}
with $\theta(p) \in \mathbb{Z}$ such that $\theta(p)=0$ for all but finitely many of the prime numbers $p$. 
Huxley \cite{Huxley1981} proved the bound $\omega_f(p^v)\leq d p^{\theta(p)/2}$ with $d=d_1+\cdots+d_g$. 
Thus we obtain the uniform bound
\begin{align}\label{eq:09292100}
\omega_f(p^v)
\leq d \sqrt{\Delta_f}. 
\end{align}
First, we consider estimates \eqref{eq:09292105} and \eqref{eq:09292108}. 
Since the von Mangoldt function $\Lambda(n)$ is supported on prime powers, we have
\begin{align}\label{eq:09292125}
\sum_{\substack{1 \leq k \leq x \\ P^+(k) \leq y}} \frac{\Lambda(k)}{k} \omega_f(k)
&= \sum_{p \leq y} \sum_{\substack{v \geq1 \\ p^v \leq x}} \frac{\log{p}}{p^v} \omega_f(p^v) \\
&= \sum_{p \leq y} \frac{\log{p}}{p} \omega_f(p)
+ \sum_{p \leq y} \sum_{\substack{v \geq2 \\ p^v \leq x}} \frac{\log{p}}{p^v} \omega_f(p^v). \nonumber
\end{align}
Using \eqref{eq:09292100}, we see that the latter term is estimated as
\begin{align}\label{eq:09292126}
\sum_{p \leq y} \sum_{\substack{v \geq2 \\ p^v \leq x}} \frac{\log{p}}{p^v} \omega_f(p^v)
\leq d \sqrt{\Delta_f} \sum_{p} \sum_{v \geq2} \frac{\log{p}}{p^v} 
\ll 1
\end{align}
with the implied constant depending only on $d$ and $\Delta_f$.
Note that we have
\begin{align*}
\omega_f(p)
\leq \omega_{f_1}(p)+\cdots+\omega_{f_g}(p)
\end{align*}
since $p \mid f(u)$ implies $p \mid f_j(u)$ for some $j \in \{1,\ldots,g\}$. 
Furthermore, it is classically known (see \cite{Nagel1921} for example) that
\begin{align}\label{eq:10101700}
\sum_{p \leq y} \frac{\log{p}}{p} \omega_{f_j}(p)
= \log{y}+O(1) 
\end{align}
for every $j \in \{1,\ldots,g\}$ with the implied constant depends only on $K_j=\mathbb{Q}(\alpha_j)$, where $\alpha_j \in \mathbb{C}$ is a solution to the equation $f_j(t)=0$. 
Hence we have 
\begin{align}\label{eq:09292127}
\sum_{p \leq y} \frac{\log{p}}{p} \omega_f(p)
\leq g\log{y}+O(1), 
\end{align}
where the implied constant depends at most on $f(t)$.
Inserting \eqref{eq:09292126} and \eqref{eq:09292127} to \eqref{eq:09292125}, we obtain \eqref{eq:09292105}. 
One can obtain \eqref{eq:09292108} along the same line, and we omit the proof. 
Next, we prove estimate \eqref{eq:09292106} as follows. 
We have 
\begin{align*}
\sum_{\substack{1 \leq k \leq x \\ P^+(k) \leq y}} \frac{(2\log{k}-\Lambda(k)) \Lambda(k)}{k} \omega_f(k)
&= \sum_{p \leq y} \sum_{\substack{v \geq1 \\ p^v \leq x}} \frac{(2v-1)(\log{p})^2}{p^v} \omega_f(p^v) \\
&= \sum_{p \leq y} \frac{(\log{p})^2}{p} \omega_f(p) +O(1)
\end{align*}
in a way similar to \eqref{eq:09292125} and \eqref{eq:09292126}. 
Here, one can check that
\begin{align*}
\sum_{p \leq y} \frac{(\log{p})^2}{p} \omega_{f_j}(p)
= \frac{1}{2} (\log{y})^2+O(\log{y})
\end{align*}
holds for every $j \in \{1,\ldots,g\}$ by the partial summation using formula \eqref{eq:10101700}. 
Hence \eqref{eq:09292106} follows by the same argument as before. 
Finally, we obtain
\begin{align*}
\sum_{\substack{1 \leq k \leq x \\ P^+(k) \leq y}} \Lambda(k) \omega_f(k)
= \sum_{p \leq y} \sum_{\substack{v \geq1 \\ p^v \leq x}} (\log{p}) \omega_f(p^v) 
&\leq d \sqrt{\Delta_f} \sum_{p \leq y} \left[\frac{\log{x}}{\log{p}}\right] (\log{p}) \\
&\ll (\log{x}) \sum_{p \leq y} 1,  
\end{align*}
where the implied constant depends only on $d$ and $\Delta_f$. 
Then we deduce that estimate \eqref{eq:09292107} holds by the prime number theorem. 
\end{proof}

Note that the function $\omega_f(k)$ is multiplicative but is not completely multiplicative in general. 
Hence we do not expect that $\omega_f(k_1 \cdots k_m)=\omega_f(k_1) \cdots \omega_f(k_m)$ holds unless $k_1,\ldots,k_m $ are pairwise coprime.\
\footnote{
It seems to be misunderstood often in \cite{Timofeev1977} that the identity $\omega_f(k_1 \cdots k_m)=\omega_f(k_1) \cdots \omega_f(k_m)$ holds for any $k_1,\ldots,k_m \in \mathbb{Z}_{\geq1}$. 
One can fix the errors with the same argument presented in this paper using Lemma \ref{lem:4.2}. 
}
Then we use the following lemma which helps us estimate multiple sums involving $\omega_f(k_1 \cdots k_m)$. 

\begin{lemma}\label{lem:4.2}
Suppose that $f(t) \in \mathbb{Z}[t]$ is a product of distinct irreducible factors $f_1(t), \ldots, f_g(t)$ over $\mathbb{Z}[t]$ of degrees $d_1, \ldots, d_g \geq1$, respectively. 
Put $d=d_1+\cdots+d_g$, and denote by $\Delta_f$ the absolute value of the discriminant of $f(t)$. 
Then we have
\begin{align*}
\omega_f(p_1^{v_1} \cdots p_m^{v_m})
\leq (d \sqrt{\Delta_f})^{m-|S|} \prod_{j \in S} \omega_f(p_j^{v_j})
\end{align*}
for any prime numbers $p_1,\ldots,p_m$ and any integers $v_1, \ldots, v_m \geq1$, where $S$ is the subset of $\{1,\ldots,m\}$ such that $j \in S$ if and only if $p_j \nmid \Delta_f$. 
\end{lemma}

\begin{proof}
Put $p_1^{v_1} \cdots p_m^{v_m}=AB$ with $A=\prod_{j \in S} p_j^{v_j}$ and $B=\prod_{j \notin S} p_j^{v_j}$. 
Since the function $\omega_f(k)$ is multiplicative, we have 
\begin{align*}
\omega_f(p_1^{v_1} \cdots p_m^{v_m})
= \omega_f(A)\omega_f(B)
\end{align*}
by noting that $(A, B)=1$. 
Here, and throughout this paper, $(A,B)$ denotes the greatest common divisor of $A$ and $B$. 
Remark that the result is trivial if $\omega_f(A)=0$. 
Thus we suppose that $\omega_f(A)\neq0$ below. 
Let $S'$ be a minimal subset of $S$ such that $\{p_j \mid j \in S \}= \{p_j \mid j \in S' \}$, and rewrite the integer $A$ as $A=\prod_{j \in S'} p_j^{w_j}$.  
In this case, the primes $p_j$ for $j \in S'$ are distinct. 
Therefore we have
\begin{align*}
\omega_f(A)
= \prod_{j \in S'} \omega_f(p_j^{w_j})
\end{align*}
by the multiplicativity of $\omega_f(k)$. 
Furthermore, we see that $\omega_f(p_j^{w_j})=\omega_f(p_j)$ holds for any $j \in S'$ by Hensel's lemma. 
Thus $\omega_f(A)$ is evaluated as
\begin{align*}
\omega_f(A)
= \prod_{j \in S'} \omega_f(p_j)
\leq \prod_{j \in S} \omega_f(p_j)
\end{align*}
since $\omega_f(p_j)\geq1$ is valid for any $j \in S$ by the assumption that $\omega_f(A)\neq0$. 
Applying Hensel's lemma again, we have $\omega_f(p_j)=\omega_f(p_j^{v_j})$ for any $j \in S$. 
Hence we conclude that the inequality
\begin{align}\label{eq:10110000}
\omega_f(A)
\leq \prod_{j \in S} \omega_f(p_j^{v_j})
\end{align}
is satisfied. 
Let $S''$ be a maximal subset of $\{1,\ldots,m\}$ with $S \subset S''$ such that $\{p_j \mid j \notin S \}= \{p_j \mid j \notin S'' \}$, and rewrite the integer $B$ as $B=\prod_{j \notin S''} p_j^{w_j}$. 
In this case, we derive
\begin{align}\label{eq:10110001}
\omega_f(B)
= \prod_{j \notin S''} \omega_f(p_j^{w_j})
\leq (d \sqrt{\Delta_f})^{m-|S''|} 
\leq (d \sqrt{\Delta_f})^{m-|S|}
\end{align}
by using \eqref{eq:09292100}. 
Combining \eqref{eq:10110000} and \eqref{eq:10110001},  we obtain the desired result. 
\end{proof}

We are now ready to evaluate $V_m^+$, $W_m^+$, and $V_i^-$, $W_i^-$ for $i=1,\ldots,m$ defined in Proposition \ref{prop:3.2}. 
Firstly, we consider the upper bounds for $V_m^+$ and $V_i^-$. 

\begin{proposition}\label{prop:4.3}
Suppose that $f(t) \in \mathbb{Z}[t]$ is a product of distinct irreducible factors $f_1(t), \ldots, f_g(t)$ over $\mathbb{Z}[t]$ of degrees $d_1, \ldots, d_g \geq1$, respectively. 
Let $x,y,z \in \mathbb{R}$ with $x>z>T_0(f)$ and $y \geq2$. 
Suppose further that $f(t) \to\infty$ as $t \to\infty$, and that $f(z)>x$. 
Put $h=x-z$. 
Then there exists a constant $A(f)>1$ depending only on $f(t)$ such that, for any $m \in \mathbb{Z}_{\geq1}$ and any $i =1,\ldots,m$, we have
\begin{align*}
V_m^+
&\leq \frac{hg^m (\log{y}+O(1))^m}{2\log{f(z)} \left(\log{\frac{f(z)}{x}}\right)^{m-1}} 
+ \frac{A(f)^m (\log{x})^m \left(\frac{y}{\log{y}}\right)^m }{\log{f(z)} \left(\log{\frac{f(z)}{x}}\right)^{m-1}}, \\
V_i^-
&\leq \frac{A(f)^i (\log{x})^{i-1} \log{f(x)} \left(\frac{y}{\log{y}}\right)^i }{\log{f(z)} \left(\log{\frac{f(z)}{x}}\right)^{i-1}}, 
\end{align*}
where $V_m^+$ and $V_i^-$ are as in Proposition \ref{prop:3.2}, and the implied constant depends only on $f(t)$.
\end{proposition}

\begin{proof}
Using \eqref{eq:09300001} with $\kappa=k_1 \cdots k_m$, we have the inequality
\begin{align*}
\sum_{\substack{z<n \leq x \\ k_1 \cdots k_m \mid f(n) \\ P^+(f(n)) \leq y}} 1
\leq \omega_f(k_1 \cdots k_m) \left(\frac{h}{k_1 \cdots k_m}+1\right). 
\end{align*}
Then $V_m^+$ is evaluated as
\begin{align}\label{eq:09301435}
V_m^+
&\leq \frac{h}{\log{f(z)} \left(\log{\frac{f(z)}{x}}\right)^{m-1}} 
\sum_{\substack{k_1,\ldots,k_m \geq 1 \\ k_1 \cdots k_m \leq h \\ \sqrt{y} < P^+(k_1) \leq y \\ P^+(k_j) \leq y ~(j \geq 2)}} 
\frac{\Lambda(k_1) \cdots \Lambda(k_m)}{k_1 \cdots k_m} \omega_f(k_1 \cdots k_m) \\
&\quad 
+ \frac{1}{\log{f(z)} \left(\log{\frac{f(z)}{x}}\right)^{m-1}} 
\sum_{\substack{k_1,\ldots,k_m \geq 1 \\ k_1 \cdots k_m \leq h \\ P^+(k_j) \leq y ~(\forall j)}} 
\Lambda(k_1) \cdots \Lambda(k_m) \omega_f(k_1 \cdots k_m). \nonumber
\end{align}
Let $S \subset \{1,\ldots,m\}$. 
Then we divide the first sum of the right-hand side of \eqref{eq:09301435} into those over $k_1, \ldots, k_m \geq1$ such that $(k_j,\Delta_f)=1$ if $j \in S$ and $(k_j,\Delta_f)>1$ otherwise, where $\Delta_f$ is the absolute value of the discriminant of $f(t)$. 
We have
\begin{align}\label{eq:09301334}
&\sum_{\substack{k_1,\ldots,k_m \geq 1 \\ k_1 \cdots k_m \leq h \\ \sqrt{y} < P^+(k_1) \leq y \\ P^+(k_j) \leq y ~(j \geq 2)}} 
= \sum_{S \subset \{1, \ldots, m\}}
\sum_{\substack{k_1,\ldots,k_m \geq 1 \\ k_1 \cdots k_m \leq h \\ \sqrt{y} < P^+(k_1) \leq y \\ P^+(k_j) \leq y ~(j \geq 2) \\ (k_j,\Delta_f)=1 \Leftrightarrow j \in S}} \\
&= \sum_{n=1}^{m} \sum_{\substack{S \subset \{1, \ldots, m\} \\ |S|=n \\ 1 \in S}}
\sum_{\substack{k_1,\ldots,k_m \geq 1 \\ k_1 \cdots k_m \leq h \\ \sqrt{y} < P^+(k_1) \leq y  \\ P^+(k_j) \leq y ~(j \geq 2) \\ (k_j,\Delta_f)=1 \Leftrightarrow j \in S}} 
+ \sum_{n=0}^{m-1} \sum_{\substack{S \subset \{1, \ldots, m\} \\ |S|=n \\ 1 \notin S}}
\sum_{\substack{k_1,\ldots,k_m \geq 1 \\ k_1 \cdots k_m \leq h \\ \sqrt{y} < P^+(k_1) \leq y  \\ P^+(k_j) \leq y ~(j \geq 2) \\ (k_j,\Delta_f)=1 \Leftrightarrow j \in S}} \nonumber
\end{align}
by dividing further the cases into $1 \in S$ or not.
Suppose that $|S|=n$ and $1 \in S$ hold. 
Since the von Mangoldt function $\Lambda(n)$ is supported on prime powers, we deduce from Lemma \ref{lem:4.2} that
\begin{align*}
&\sum_{\substack{k_1,\ldots,k_m \geq 1 \\ k_1 \cdots k_m \leq h \\ \sqrt{y} < P^+(k_1) \leq y  \\ P^+(k_j) \leq y ~(j \geq 2) \\ (k_j,\Delta_f)=1 \Leftrightarrow j \in S}} 
\frac{\Lambda(k_1) \cdots \Lambda(k_m)}{k_1 \cdots k_m} \omega_f(k_1 \cdots k_m) \\
&\leq \sum_{\substack{k_1,\ldots,k_m \geq 1 \\ k_1 \cdots k_m \leq h \\ \sqrt{y} < P^+(k_1) \leq y  \\ P^+(k_j) \leq y ~(j \geq 2) \\ (k_j,\Delta_f)=1 \Leftrightarrow j \in S}} 
\frac{\Lambda(k_1) \cdots \Lambda(k_m)}{k_1 \cdots k_m} (d \sqrt{\Delta_f})^{m-n} \prod_{j \in S} \omega_f(k_j) \\
&\leq \sum_{\substack{1 \leq k \leq x \\ \sqrt{y}< P^+(k) \leq y}} \frac{\Lambda(k)}{k} \omega_f(k)
\left(\sum_{\substack{1 \leq k \leq x \\ P^+(k) \leq y}} \frac{\Lambda(k)}{k} \omega_f(k)\right)^{n-1} 
\left(d \sqrt{\Delta_f} \sum_{\substack{1 \leq k \leq x \\ P^+(k) \leq y \\ (k,\Delta_f)>1}} \frac{\Lambda(k)}{k} \right)^{m-n} 
\end{align*}
by recalling that $h=x-z<x$. 
Applying \eqref{eq:09292105} and \eqref{eq:09292108} of Lemma \ref{lem:4.1}, we have
\begin{align*}
&\sum_{\substack{1 \leq k \leq x \\ \sqrt{y}< P^+(k) \leq y}} \frac{\Lambda(k)}{k} \omega_f(k)
\left(\sum_{\substack{1 \leq k \leq x \\ P^+(k) \leq y}} \frac{\Lambda(k)}{k} \omega_f(k)\right)^{n-1} \\
&\leq \left(\frac{g}{2} \log{y}+O(1)\right)
\left(g\log{y}+O(1)\right)^{n-1}
= \frac{1}{2} \left(g \log{y}+O(1)\right)^n, 
\end{align*}
where the implied constants depend only on $f(t)$. 
Then we consider the sum over $k$ with $(k,\Delta_f)>1$. 
We may suppose that $k$ varies over prime powers since $\Lambda(k)=0$ otherwise. 
Furthermore, if $k=p^v$ satisfies $(k,\Delta_f)>1$, then $p \mid \Delta_f$ must be satisfied. 
Therefore we obtain
\begin{align*}
\sum_{\substack{1 \leq k \leq x \\ P^+(k) \leq y \\ (k,\Delta_f)>1}} \frac{\Lambda(k)}{k}
= \sum_{\substack{p \leq y \\ p \mid \Delta_f}} \sum_{\substack{v \geq1 \\ p^v \leq x}} \frac{\log{p}}{p^v}
\leq \sum_{p \mid \Delta_f} \frac{\log{p}}{p-1}
\leq \log{\Delta_f}. 
\end{align*}
Hence we arrive at the estimate
\begin{align}\label{eq:09301335}
&\sum_{\substack{k_1,\ldots,k_m \geq 1 \\ k_1 \cdots k_m \leq h \\ \sqrt{y} < P^+(k_1) \leq y  \\ P^+(k_j) \leq y ~(j \geq 2) \\ (k_j,\Delta_f)=1 \Leftrightarrow j \in S}} 
\frac{\Lambda(k_1) \cdots \Lambda(k_m)}{k_1 \cdots k_m} \omega_f(k_1 \cdots k_m) \\
&\leq \frac{1}{2} \left(g \log{y}+O(1)\right)^n \left(d \sqrt{\Delta_f} \log{\Delta_f}\right)^{m-n} \nonumber
\end{align}
if $|S|=n$ and $1 \in S$. 
On the other hand, if we suppose that $|S|=n$ and $1 \notin S$, then we have
\begin{align}\label{eq:09301336}
&\sum_{\substack{k_1,\ldots,k_m \geq 1 \\ k_1 \cdots k_m \leq h \\ \sqrt{y} < P^+(k_1) \leq y  \\ P^+(k_j) \leq y ~(j \geq 2) \\ (k_j,\Delta_f)=1 \Leftrightarrow j \in S}} 
\frac{\Lambda(k_1) \cdots \Lambda(k_m)}{k_1 \cdots k_m} \omega_f(k_1 \cdots k_m) \\
&\leq \left(\sum_{\substack{1 \leq k \leq x \\ P^+(k) \leq y}} \frac{\Lambda(k)}{k} \omega_f(k)\right)^{n} 
\left(d \sqrt{\Delta_f} \sum_{\substack{1 \leq k \leq x \\ P^+(k) \leq y \\ (k,\Delta_f)>1}} \frac{\Lambda(k)}{k} \right)^{m-n} \nonumber\\
&\leq \left(g \log{y}+O(1)\right)^n \left(d \sqrt{\Delta_f} \log{\Delta_f}\right)^{m-n} \nonumber
\end{align}
in a similar way. 
Here, the number of the subsets $S \subset \{1,\ldots,m\}$ with $|S|=n$ and $1 \in S$ equals to $\binom{m-1}{n-1}$, and the number of $S \subset \{1,\ldots,m\}$ with $|S|=n$ and $1 \notin S$ equals to $\binom{m-1}{n}$. 
Therefore, by \eqref{eq:09301334}, \eqref{eq:09301335}, and \eqref{eq:09301336}, we derive 
\begin{align}\label{eq:09301526}
&\sum_{\substack{k_1,\ldots,k_m \geq 1 \\ k_1 \cdots k_m \leq h \\ \sqrt{y} < P^+(k_1) \leq y  \\ P^+(k_j) \leq y ~(j \geq 2)}} 
\frac{\Lambda(k_1) \cdots \Lambda(k_m)}{k_1 \cdots k_m} \omega_f(k_1 \cdots k_m) \\
&\leq \sum_{n=1}^{m} \binom{m-1}{n-1} \frac{1}{2} \left(g \log{y}+O(1)\right)^{n} \left(d \sqrt{\Delta_f} \log{\Delta_f}\right)^{m-n} \nonumber \\
&\quad
+ \sum_{n=0}^{m-1} \binom{m-1}{n} \left(g \log{y}+O(1)\right)^{n} \left(d \sqrt{\Delta_f} \log{\Delta_f}\right)^{m-n} \nonumber \\
&= \frac{1}{2}  \sum_{n=0}^{m} \binom{m}{n} \left(g \log{y}+O(1)\right)^{n} \left(d \sqrt{\Delta_f} \log{\Delta_f}\right)^{m-n} \nonumber \\
&\quad
+ \frac{1}{2} \sum_{n=0}^{m-1} \binom{m-1}{n} \left(g \log{y}+O(1)\right)^{n} \left(d \sqrt{\Delta_f} \log{\Delta_f}\right)^{m-n} \nonumber \\
&= \frac{1}{2} (g \log{y}+O(1))^m 
+ \frac{1}{2} d \sqrt{\Delta_f} \log{\Delta_f} (g \log{y}+O(1))^{m-1} \nonumber \\
&= \frac{1}{2} g^m(\log{y}+O(1))^m , \nonumber
\end{align}
where the implied constants depend only on $f(t)$.
Next, we consider the second sum of the right-hand side of \eqref{eq:09301435}. 
We have
\begin{align*}
\sum_{\substack{1 \leq k \leq x \\ P^+(k) \leq y \\ (k,\Delta_f)>1}} \Lambda(k)
= \sum_{\substack{p \leq y \\ p \mid \Delta_f}} \sum_{\substack{v \geq1 \\ p^v \leq x}} \log{p}
\leq \sum_{p \mid \Delta_f} \left[\frac{\log{x}}{\log{p}}\right] \log{p}
\ll \log{x},
\end{align*}
where the implied constant depends only on $\Delta_f$.
Using this and \eqref{eq:09292107} of Lemma \ref{lem:4.1} in place of \eqref{eq:09292105} and \eqref{eq:09292108}, we also obtain
\begin{align}\label{eq:09301344}
&\sum_{\substack{k_1,\ldots,k_m \geq 1 \\ k_1 \cdots k_m \leq h \\ P^+(k_j) \leq y ~(\forall j)}} 
\Lambda(k_1) \cdots \Lambda(k_m) \omega_f(k_1 \cdots k_m) \\
&= \sum_{n=0}^{m} \sum_{\substack{S \subset \{1, \ldots, m\} \\ |S|=n}}
\sum_{\substack{k_1,\ldots,k_m \geq 1 \\ k_1 \cdots k_m \leq h \\ P^+(k_j) \leq y ~(\forall j) \\ (k_j,\Delta_f)=1 \Leftrightarrow j \in S}} 
\Lambda(k_1) \cdots \Lambda(k_m) \omega_f(k_1 \cdots k_m) \nonumber \\
&\leq \sum_{n=0}^{m} \binom{m}{n} \left(C_1(f)y\frac{\log{x}}{\log{y}}\right)^n \left(d \sqrt{\Delta_f} C_2(f) \log{x}\right)^{m-n} \nonumber\\
&= \left(C_1(f) y\frac{\log{x}}{\log{y}} 
+ d \sqrt{\Delta_f} C_2(f) \log{x}\right)^m \nonumber  \\
&\leq A(f)^m (\log{x})^m \left(\frac{y}{\log{y}}\right)^m, \nonumber 
\end{align}
where $C_i(f)$ and $A(f)$ are positive constants depending at most on $f(t)$.
As a result, the desired estimate for $V_m^+$ follows from \eqref{eq:09301435}. 
As for $V_i^-$, we have
\begin{align*}
&\sum_{\substack{k_1,\ldots,k_i \geq 1 \\ k_1 \cdots k_{i-1} \leq h,~ \frac{h}{k_1 \cdots k_{i-1}}<k_i \leq f(x) \\ P^+(k_j) \leq y ~(\forall j)}} 
\Lambda(k_1) \cdots \Lambda(k_i) \omega_f(k_1 \cdots k_i) \\
&\leq A(f)^i (\log{x})^{i-1} \log{f(x)} \left(\frac{y}{\log{y}}\right)^{i} 
\end{align*}
for every $i=1,\ldots,m$ along the same line as \eqref{eq:09301344}. 
This yields the desired estimate for $V_i^-$. 
From the above, the proof of the results are completed.
\end{proof}

Secondly, we consider the upper bounds for $W_m^+$ and $W_i^-$. 

\begin{proposition}\label{prop:4.4}
Suppose that $f(t) \in \mathbb{Z}[t]$ satisfies the same conditions as in Proposition \ref{prop:4.3}. 
Let $x,y,z \in \mathbb{R}$ with $x>z>T_0(f)$ and $y \geq2$. 
Put $h=x-z$. 
Then there exists a constant $A(f)>1$ depending only on $f(t)$ such that, for any $m \in \mathbb{Z}_{\geq1}$ and any $i =1,\ldots,m$, we have 
\begin{align*}
W_m^+
&\leq \frac{hg^m (2g+1) (\log{y}+O(1))^{m+1}}{8(\log{f(z)})^2 \left(\log{\frac{f(z)}{x}}\right)^{m-1}} 
+ \frac{A(f)^{m+1} (\log{x})^{m+1} \left(\frac{y}{\log{y}}\right)^{m+1} }{(\log{f(z)})^2 \left(\log{\frac{f(z)}{x}}\right)^{m-1}}, \\
W_i^-
&\leq \frac{A(f)^{i+1} (\log{x})^{i-1} (\log{f(x)})^2 \left(\frac{y}{\log{y}}\right)^{i+1} }{(\log{f(z)})^2 \left(\log{\frac{f(z)}{x}}\right)^{i-1}} 
\end{align*}
where $W_m^+$ and $W_i^-$ are as in Proposition \ref{prop:3.2}, and the implied constant depends only on $f(t)$. 
\end{proposition}

\begin{proof}
Using \eqref{eq:09300001} with $\kappa=[k_1,k_2]k_3 \cdots k_m$, we obtain
\begin{align*}
W_m^+
&\leq \frac{h}{(\log{f(z)})^2 \left(\log{\frac{f(z)}{x}}\right)^{m-1}} \\
&\quad \times
\sum_{\substack{k_1,\ldots,k_{m+1} \geq 1 \\ [k_1,k_2]k_3 \cdots k_{m+1} \leq h \\ P^+(k_j) \leq \sqrt{y} ~(j=1,2) \\ P^+(k_j) \leq y ~(j \geq 3)}} 
\frac{\Lambda(k_1) \cdots \Lambda(k_{m+1})}{[k_1,k_2] k_3 \cdots k_{m+1}} \omega_f([k_1,k_2] k_3 \cdots k_{m+1}) \nonumber\\
&\quad 
+ \frac{1}{(\log{f(z)})^2 \left(\log{\frac{f(z)}{x}}\right)^{m-1}} \nonumber\\
&\qquad \times
\sum_{\substack{k_1,\ldots,k_{m+1} \geq 1 \\ [k_1,k_2]k_3 \cdots k_{m+1} \leq h \\ P^+(k_j) \leq y ~(\forall j)}} 
\Lambda(k_1) \cdots \Lambda(k_{m+1}) \omega_f([k_1,k_2] k_3 \cdots k_{m+1}) \nonumber
\end{align*}
by the definition of $W_m^+$. 
We have
\begin{align*}
\sum_{\substack{k_1,\ldots,k_{m+1} \geq 1 \\ [k_1,k_2] k_3 \cdots k_{m+1} \leq h \\ P^+(k_j) \leq \sqrt{y} ~(j=1,2) \\ P^+(k_j) \leq y ~(j \geq 3)}} 
= \sum_{\substack{k_1,\ldots,k_{m+1} \geq 1 \\ [k_1,k_2] k_3 \cdots k_{m+1} \leq h \\ P^+(k_j) \leq \sqrt{y} ~(j=1,2) \\ P^+(k_j) \leq y ~(j \geq 3) \\ (k_1,k_2)=1}} 
+ \sum_{\substack{k_1,\ldots,k_{m+1} \geq 1 \\ [k_1,k_2] k_3 \cdots k_{m+1} \leq h \\ P^+(k_j) \leq \sqrt{y} ~(j=1,2) \\ P^+(k_j) \leq y ~(j \geq 3) \\ (k_1,k_2)>1}}.  
\end{align*}
Then we divide the sums into those over $k_1, \ldots, k_{m+1} \geq1$ such that $(k_j,\Delta_f)=1$ if $j \in S$ and $(k_j,\Delta_f)>1$ otherwise, where $S$ is a subset of $\{1,\ldots,m+1\}$. 
Recall that $[k_1,k_2] k_3 \cdots k_{m+1}=k_1k_2k_3 \cdots k_{m+1}$ holds if $(k_1,k_2)=1$. 
Therefore we obtain
\begin{align}\label{eq:09301657}
&\sum_{\substack{k_1,\ldots,k_{m+1} \geq 1 \\ [k_1,k_2] k_3 \cdots k_{m+1} \leq h \\ P^+(k_j) \leq \sqrt{y} ~(j=1,2) \\ P^+(k_j) \leq y ~(j \geq 3) \\ (k_1,k_2)=1}} 
\frac{\Lambda(k_1) \cdots \Lambda(k_{m+1})}{[k_1,k_2] k_3 \cdots k_{m+1}} \omega_f([k_1,k_2] k_3 \cdots k_{m+1}) \\
&\leq \sum_{n=2}^{m} \binom{m-2}{n-2} \frac{1}{4} \left(g \log{y}+O(1)\right)^{n} \left(d \sqrt{\Delta_f} \log{\Delta_f}\right)^{m-n} \nonumber \\
&\quad
+ 2\sum_{n=1}^{m-1} \binom{m-2}{n-1} \frac{1}{2} \left(g \log{y}+O(1)\right)^{n} \left(d \sqrt{\Delta_f} \log{\Delta_f}\right)^{m-n} \nonumber\\
&\quad
+ \sum_{n=0}^{m-2} \binom{m-2}{n} \left(g \log{y}+O(1)\right)^{n} \left(d \sqrt{\Delta_f} \log{\Delta_f}\right)^{m-n} \nonumber\\
&= \frac{1}{4} g^{m+1} (\log{y}+O(1))^{m+1} \nonumber
\end{align}
by an argument similar to \eqref{eq:09301526}, where we divide the cases into the followings: 
\begin{itemize}
\item[$(\mathrm{a})$]
$1,2 \in S$;
\item[$(\mathrm{b})$]
$1 \in S$ and $2 \notin S$, or $1 \notin S$ and $2 \in S$; 
\item[$(\mathrm{c})$]
$1,2 \notin S$. 
\end{itemize}
On the other hand, the sum over $k_1, \ldots,k_{m+1}$ such that $(k_1,k_2)>1$ is estimated as follows. 
We have  
\begin{align*}
&\sum_{\substack{k_1,\ldots,k_{m+1} \geq 1 \\ [k_1,k_2]k_3 \cdots k_{m+1} \leq h \\ P^+(k_j) \leq \sqrt{y} ~(j=1,2) \\ P^+(k_j) \leq y ~(j \geq 3) \\ (k_1,k_2)>1}} 
\frac{\Lambda(k_1) \cdots \Lambda(k_{m+1})}{[k_1,k_2]k_3 \cdots k_{m+1}} \omega_f([k_1,k_2]k_3 \cdots k_{m+1}) \\
&= \sum_{\substack{\kappa,k_3,\ldots,k_{m+1} \geq 1 \\ \kappa k_3 \cdots k_{m+1} \leq h \\ P^+(\kappa) \leq \sqrt{y} \\ P^+(k_j) \leq y ~(j \geq 3) }}
\sum_{\substack{k_1,k_2 \geq1 \\ [k_1,k_2]=\kappa \\ (k_1,k_2)>1}} \frac{\Lambda(k_1)\Lambda(k_2)}{[k_1,k_2]}
\frac{\Lambda(k_3) \cdots \Lambda(k_{m+1})}{k_3 \cdots k_{m+1}} \omega_f(\kappa k_3 \cdots k_{m+1}). 
\end{align*}
Furthermore, since the von Mangoldt function $\Lambda(n)$ is supported on prime powers, the inner sum is calculated as
\begin{align*}
\sum_{\substack{k_1,k_2 \geq1 \\ [k_1,k_2]=\kappa \\ (k_1,k_2)>1}} \frac{\Lambda(k_1)\Lambda(k_2)}{[k_1,k_2]}
&= \frac{\Lambda(\kappa)}{\kappa} \sum_{k_1 \mid \kappa} \Lambda(k_1)
+ \frac{\Lambda(\kappa)}{\kappa} \sum_{k_2 \mid \kappa} \Lambda(k_2)
- \frac{\Lambda(\kappa)^2}{\kappa} \\
&= \frac{(2\log{\kappa}-\Lambda(\kappa)) \Lambda(\kappa)}{\kappa}
\end{align*}
by recalling the formula $\sum_{d \mid \kappa} \Lambda(d)=\log{\kappa}$. 
Hence, applying \eqref{eq:09292105} and \eqref{eq:09292106}, we find that the estimate
\begin{align}\label{eq:10012056}
&\sum_{\substack{k_1,\ldots,k_{m+1} \geq 1 \\ [k_1,k_2]k_3 \cdots k_{m+1} \leq h \\ P^+(k_j) \leq \sqrt{y} ~(j=1,2) \\ P^+(k_j) \leq y ~(j \geq 3) \\ (k_1,k_2)>1}} 
\frac{\Lambda(k_1) \cdots \Lambda(k_{m+1})}{[k_1,k_2]k_3 \cdots k_{m+1}} \omega_f([k_1,k_2]k_3 \cdots k_{m+1}) \\
&\leq \frac{1}{8} g^m (\log{y}+O(1))^{m+1} \nonumber 
\end{align}
holds with the same argument as in \eqref{eq:09301657}. 
By \eqref{eq:09301657} and \eqref{eq:10012056}, we derive
\begin{align*}
&\sum_{\substack{k_1,\ldots,k_{m+1} \geq 1 \\ [k_1,k_2]k_3 \cdots k_{m+1} \leq h \\ P^+(k_j) \leq \sqrt{y} ~(j=1,2) \\ P^+(k_j) \leq y ~(j \geq 3)}} 
\frac{\Lambda(k_1) \cdots \Lambda(k_{m+1})}{[k_1,k_2]k_3 \cdots k_{m+1}} \omega_f([k_1,k_2]k_3 \cdots k_{m+1}) \\
&\leq \frac{1}{8} g^m(2g+1) (\log{y}+O(1))^{m+1}. 
\end{align*}
One can also obtain
\begin{align}\label{eq:10012104}
&\sum_{\substack{k_1,\ldots,k_{m+1} \geq 1 \\ [k_1,k_2]k_3 \cdots k_{m+1} \leq h \\ P^+(k_j) \leq y ~(\forall j)}} 
\Lambda(k_1) \cdots \Lambda(k_{m+1}) \omega_f([k_1,k_2] k_3 \cdots k_{m+1}) \\
&\leq A(f)^{m+1} (\log{x})^{m+1} \left(\frac{y}{\log{y}}\right)^{m+1} \nonumber
\end{align}
in a similar way using \eqref{eq:09292107} in place of \eqref{eq:09292105} and \eqref{eq:09292106}. 
From these, we conclude that the desired estimate for $W_m^+$ follows. 
Finally, we have for $2 \leq i \leq m$, 
\begin{align*}
&\sum_{\substack{k_1,\ldots,k_{i+1} \geq 1 \\ [k_1,k_2]k_3 \cdots k_i \leq h,~ \frac{h}{[k_1,k_2]k_3 \cdots k_i}<k_{i+1} \leq f(x) \\ P^+(k_j) \leq y ~(\forall j)}} 
\Lambda(k_1) \cdots \Lambda(k_{i+1}) \omega_f([k_1,k_2] k_3 \cdots k_{i+1}) \\
&\leq A(f)^{i+1} (\log{x})^i \log{f(x)} \left(\frac{y}{\log{y}}\right)^{i+1} 
\end{align*}
along the same line as \eqref{eq:10012104}. 
Therefore $W_i^-$ is estimated as
\begin{align*}
W_i^-
\leq \frac{A(f)^{i+1} (\log{x})^i \log{f(x)} \left(\frac{y}{\log{y}}\right)^{i+1}}{(\log{f(z)})^2 \left(\log{\frac{f(z)}{x}}\right)^{i-1}} 
\end{align*}
for $2 \leq i \leq m$. 
Furthermore, one can check that the estimate
\begin{align*}
W_1^-
\leq \frac{A(f)^{2} (\log{f(x)})^2 \left(\frac{y}{\log{y}}\right)^2}{(\log{f(z)})^2}
\end{align*}
holds. 
Hence the desired estimate for $W_i^-$ holds for every $i=1,\ldots,m$ by recalling that $\log{f(x)}>\log{f(z)}>\log{x}$. 
\end{proof}

\begin{proof}[Proof of Theorem \ref{thm:1.1}]
Recall that we may suppose that $f(t)$ satisfies the conditions of Propositions \ref{prop:4.3} and \ref{prop:4.4} without loss of generality. 
Then we complete the proof by suitably choosing the parameters in Propositions \ref{prop:4.3} and \ref{prop:4.4}. 
Let
\begin{align}\label{eq:10021617}
y
= x^{1/u}
\quad\text{and}\quad
m
= [u]
\quad\text{with}\quad
1 \leq u
\leq \frac{\sqrt{\log{x}}}{\log\log{x}}. 
\end{align}
Then we have
\begin{align*}
h g^m (\log{y}+O(1))^m
&\leq x g^{[u]} (\log{x})^{[u]} u^{-[u]} \left(1+O\left(\frac{u}{\log{x}}\right)\right)^{[u]} \\
&= x g^{[u]} (\log{x})^{[u]} u^{-[u]} \left(1+O\left(\frac{1}{(\log\log{x})^2}\right)\right) 
\end{align*}
for $x \geq x_0(f)$ with a large positive constant $x_0(f)$. 
On the other hand, we obtain
\begin{align}\label{eq:10021642}
A(f)^m (\log{x})^m \left(\frac{y}{\log{y}}\right)^m
&\leq x A(f)^{[u]} u^{[u]} \\
&= x g^{[u]} (\log{x})^{[u]} u^{-[u]} \left(\frac{A(f)}{g} \frac{u^2}{\log{x}}\right)^{[u]} \nonumber\\
&\leq x g^{[u]} (\log{x})^{[u]} u^{-[u]} \frac{A(f)}{(\log\log{x})^2} \nonumber
\end{align}
for $x \geq \exp(\exp(\sqrt{A(f)}))$.
Therefore, Proposition \ref{prop:4.3} implies the upper bound
\begin{align}\label{eq:10021637}
V_m^+
\leq \frac{x g^{[u]} (\log{x})^{[u]} u^{-[u]}}{2\log{f(z)} \left(\log{\frac{f(z)}{x}}\right)^{m-1}} 
\left(1+O\left(\frac{1}{(\log\log{x})^2}\right)\right)
\end{align}
for sufficiently large $x$, provided the conditions $f(z)>x$ and $x>z>T_0(f)$ are satisfied. 
Then we choose the parameter $z$ as 
\begin{align}\label{eq:10022046}
z 
= \frac{g^{[u]}}{d(d-1)^{[u]-1}u^{[u]}} \frac{x}{\log{x}}. 
\end{align}
One can see from the condition on $u$ in \eqref{eq:10021617} that $\sqrt{x}<z<4x (\log{x})^{-1}$ is satisfied for sufficiently large $x$. 
Therefore the condition $x>z>T_0(f)$ holds by making $x$ large enough. 
To check that $f(z)>x$ holds, we factorize the polynomial $f(t)$ as $f(t)=c_f(t-\alpha_1) \cdots (t-\alpha_d)$ with $c_f \in \mathbb{Z}_{>0}$ and $\alpha_1,\ldots,\alpha_d \in \mathbb{C}$. 
Then $\log{f(t)}$ is evaluated as 
\begin{align*}
\log{f(t)}
= \log{c_f}
+ d \log{t} 
+ \sum_{j=1}^{d} \log\left|1-\frac{\alpha_j}{t}\right| 
= d \log{t}+O(1)
\end{align*}
for $t>\max\{T_0(f),|\alpha_1|,\ldots,|\alpha_d|\}$, where the implied constant depends only on $c_f$ and $d$. 
Applying this, we derive the asymptotic formula
\begin{align*}
\log{f(z)}
&= d \log{x} 
+ d \log\left(\frac{g^{[u]}}{d(d-1)^{[u]-1}u^{[u]}}\right)
+ O(\log\log{x}) \\ 
&= d \log{x}
+ O\left(u (\log{u}+1) +\log\log{x}\right), 
\end{align*}
where the implied constant depends at most on $f(t)$. 
This yields the formulas
\begin{align}
\log{f(z)}
&= d \log{x} \left(1+O\left(\frac{1}{\sqrt{\log{x}}}\right)\right), \label{eq:10031746} \\
\log{\frac{f(z)}{x}}
&= (d-1) \log{x} \left(1+O\left(\frac{1}{\sqrt{\log{x}}}\right)\right) \label{eq:10031747} 
\end{align}
by the condition on $u$ in \eqref{eq:10021617}. 
Here, we note that \eqref{eq:10031746} implies that $f(z)>x$ holds for sufficiently large $x$ due to $d \geq2$. 
Furthermore, inserting \eqref{eq:10031746} and \eqref{eq:10031747} to \eqref{eq:10021637}, we arrive at
\begin{align*}
V_m^+
\leq \frac{x g^{[u]}}{2d(d-1)^{[u]-1} u^{[u]}} 
\left(1+O\left(\frac{1}{\log\log{x}}\right)\right) 
\end{align*}
by recalling that we have
\begin{align*}
\left(1+O\left(\frac{1}{\sqrt{\log{x}}}\right)\right)^m
= 1+O\left(\frac{1}{\log\log{x}}\right)
\end{align*}
for sufficiently large $x$. 
Next, we evaluate $V_i^-$ for $i=1,\ldots,m$ as follows. 
In a way similar to \eqref{eq:10021642}, we obtain
\begin{align*}
A(f)^i (\log{x})^{i-1} \log{f(x)} \left(\frac{y}{\log{y}}\right)^i
\ll g^i x^{i/u}  (\log{x})^i u^{-i} \frac{1}{(\log\log{x})^2}, 
\end{align*}
where the implied constant depends at most on $f(t)$. 
Applying Proposition \ref{prop:4.3}, we derive the estimate
\begin{align*}
V_i^-
\ll \frac{g^i x^{i/u}  (\log{x})^i u^{-i}}{\log{f(z)} \left(\log{\frac{f(z)}{x}}\right)^{i-1}}
\frac{1}{(\log\log{x})^2}
\ll \frac{d-1}{d} \left(\frac{g}{d-1}\frac{x^{1/u}}{u}\right)^i 
\frac{1}{(\log\log{x})^2}
\end{align*}
by using \eqref{eq:10031746} and \eqref{eq:10031747}, where the implied constant is independent of $i$ and depends at most on $f(t)$.  
Remark that 
\begin{align*}
\frac{x^{1/u}}{u}
= \exp\left(\frac{1}{u}\log{x}-\log{u}\right)
\geq \exp\left(\sqrt{\log{x}}\right)
> 2
\end{align*}
holds for any $x \geq2$ by \eqref{eq:10021617}. 
It yields
\begin{align*}
\sum_{i=1}^{m} V_i^-
&\ll \frac{d-1}{d} \left(\frac{g}{d-1}\frac{x^{1/u}}{u}\right)^m
\frac{1}{(\log\log{x})^2} \\
&\leq \frac{x g^{[u]}}{d(d-1)^{[u]-1} u^{[u]}} \frac{1}{(\log\log{x})^2} \nonumber 
\end{align*} 
since $A+A^2+\cdots+A^m<2A^m$ for $A>2$. 
The bounds for $W_m^+$ and $W_i^-$ can be obtained similarly. 
Indeed, one can see that
\begin{align*}
W_m^+
&\leq \frac{x g^{[u]}(2g+1)}{8d^2(d-1)^{[u]-1} u^{[u]+1}} 
\left(1+O\left(\frac{1}{\log\log{x}}\right)\right), \\
\sum_{i=1}^{m} W_i^-
&\ll \frac{x g^{[u]}(2g+1)}{d^2(d-1)^{[u]-1} u^{[u]+1}} 
\frac{1}{(\log\log{x})^2}
\end{align*}
for sufficiently large $x$. 
As a result, we conclude that inequality \eqref{eq:09272331} holds with 
\begin{align*}
V
&= \frac{x g^{[u]}}{2d(d-1)^{[u]-1} u^{[u]}}
\left(1+O\left(\frac{1}{\log\log{x}}\right)\right), \\
W
&= \frac{x g^{[u]}(2g+1)}{8d^2(d-1)^{[u]-1} u^{[u]+1}} 
\left(1+O\left(\frac{1}{\log\log{x}}\right)\right)
\end{align*}
by Proposition \ref{prop:3.2}. 
Therefore, by \eqref{eq:10022027}, we obtain
\begin{align*}
&\Psi_f(x,x^{1/u})-\Psi_f(z,x^{1/u}) \\
&< \left(\frac{1}{2}+\frac{1}{2}\frac{2g+1}{8du}+\sqrt{\frac{1}{2}\frac{2g+1}{8du}+\frac{1}{4}\left(\frac{2g+1}{8du}\right)^2}\right) \\
&\qquad \times
\frac{x g^{[u]}}{d(d-1)^{[u]-1} u^{[u]}} 
\left(1+O\left(\frac{1}{\log\log{x}}\right)\right) \\
&= \gamma_f(u) \frac{g^{[u]}}{d(d-1)^{[u]-1} u^{[u]}} 
\left(x+O\left(\frac{x}{\log\log{x}}\right)\right). 
\end{align*}
We have trivially $\Psi_f(z,x^{1/u})\leq z$. 
Hence the desired bound for $\Psi_f(x,x^{1/u})$ follows by the choice of $z$ as in \eqref{eq:10022046}. 
\end{proof}

\section{Proof of Theorem \ref{thm:1.2} and its corollary}\label{sec:5}
Let $\alpha$ be an arbitrary algebraic irrational number, and $K=\mathbb{Q}(\alpha)$ the algebraic field generated by $\alpha$ over the rationals. 
For any rational integer $n \geq0$, we have
\begin{gather*}
(n+\alpha) \mathfrak{a}
= \prod_{\mathfrak{p} \in J_\alpha} \mathfrak{p}^{u_n(\mathfrak{p})}, 
\end{gather*}
where $J_\alpha$ denotes the set of all prime ideals of $K$, and $u_n(\mathfrak{p})$ are non-negative integers. 
Define $P_\alpha$ as the subset of $J_\alpha$ consisting of prime ideals $\mathfrak{p}$ such that 
\begin{itemize}
\item[$(\mathrm{i})$]
$\mathfrak{p}$ is of the first degree; 
\item[$(\mathrm{ii})$]
$\mathfrak{p}$ is unambiguous. 
\end{itemize}
In other words, we have $\norm(\mathfrak{p})=p$ and $\mathfrak{p}^2 \nmid (p)$ for a rational prime $p$ if $\mathfrak{p} \in P_\alpha$. 
Then we rewrite the factorization of $(n+\alpha) \mathfrak{a}$ as 
\begin{align*}
(n+\alpha) \mathfrak{a}
= \mathfrak{b}_n \prod_{\mathfrak{p} \in P_\alpha} \mathfrak{p}^{u_n(\mathfrak{p})}, 
\end{align*}
where $\mathfrak{b}_n$ is the integral containing all prime factors of $(n+\alpha) \mathfrak{a}$ which are not in $P_\alpha$. 
By \cite[Lemma 4.4]{Mine2022+}, we know that there exists a constant $C(\alpha)>0$ depending at most on $\alpha$ such that
\begin{align}\label{eq:10061552}
\norm(\mathfrak{b}_n)
\leq C(\alpha)
\end{align}
holds for any $n \in \mathbb{Z}$. 
Thus we have $\norm(\mathfrak{p}) \leq C(\alpha)$ for any $\mathfrak{p} \in J_\alpha \setminus P_\alpha$. 
Then we prove Theorem \ref{thm:1.2} by using the following lemmas. 

\begin{lemma}\label{lem:5.2}
Take an irreducible polynomial $f(t) \in \mathbb{Z}[t]$ with $f(-\alpha)=0$, and denote by $A_f$ its leading coefficient. 
Let $C(\alpha)>0$ be the constant described in \eqref{eq:10061552}. 
Then, for any $n \in \mathbb{Z}$ and any $x>\max\{C(\alpha), \norm(\mathfrak{a}), |A_f|\}$, the following conditions are equivalent. 
\begin{itemize}
\item[$(\mathrm{A})$]
There exists a rational prime $p$ such that $p \mid f(n)$ and $p>x$. 
\item[$(\mathrm{B})$]
There exists a prime ideal $\mathfrak{p} \in P_\alpha$ such that $\mathfrak{p} \mid (n+\alpha)\mathfrak{a}$ and $\norm(\mathfrak{p})>x$. 
\end{itemize}
\end{lemma}

\begin{proof}
We first remark that the identity
\begin{align}\label{eq:10031102}
\frac{\norm((n+\alpha)\mathfrak{a})}{\norm(\mathfrak{a})}
= \left|(n+\alpha_1) \cdots (n+\alpha_d)\right|
= \left|\frac{f(n)}{A_f}\right|
\end{align}
holds, where $\alpha_1,\ldots,\alpha_d$ are the conjugates of $\alpha$ over $\mathbb{Q}$. 
Suppose that there exists a rational prime $p$ such that $p \mid f(n)$ and $p>x$. 
By \eqref{eq:10031102}, we deduce that $p$ divides $\norm((n+\alpha)\mathfrak{a})$ or $A_f$, but the later does not hold due to $x>|A_f|$. 
Thus we have that $p \mid \norm((n+\alpha)\mathfrak{a})$. 
Furthermore, it implies that there exists a prime ideal $\mathfrak{p}$ of $K$ such that $\mathfrak{p} \mid (n+\alpha)\mathfrak{a}$ and $\norm(\mathfrak{p})=p^f$ with some $f \geq1$. 
If we suppose that $\mathfrak{p} \notin P_\alpha$, then $\norm(\mathfrak{p})>x>C(\alpha)$ would follow, which yields contradiction. 
Hence we conclude that $\mathfrak{p} \in P_\alpha$ holds, and therefore, $(\mathrm{A})$ implies $(\mathrm{B})$. 

Suppose conversely that there exists a prime ideal $\mathfrak{p} \in P_\alpha$ such that $\mathfrak{p} \mid (n+\alpha)\mathfrak{a}$ and $\norm(\mathfrak{p})>x$. 
Then $\norm(\mathfrak{p})=p$ holds for a rational prime $p$ by the definition of $P_\alpha$. 
Thus $p \mid \norm((n+\alpha)\mathfrak{a})$ and $p>x$ are satisfied. 
Furthermore, $p \mid \norm((n+\alpha)\mathfrak{a})$ implies that $p$ divides $f(n)$ or $\norm(\mathfrak{a})$ by \eqref{eq:10031102}, but the later does not hold due to $x>\norm(\mathfrak{a})$. 
As a result, we conclude that $(\mathrm{B})$ implies $(\mathrm{A})$. 
\end{proof}

\begin{lemma}[Lemma 4.6 in \cite{Mine2022+}]\label{lem:5.3}
Let $\mathfrak{p} \in P_\alpha$ and $v \geq1$. 
Assume $\mathfrak{p}^v \mid (m+\alpha)\mathfrak{a}$ and $\mathfrak{p}^v \mid (n+\alpha)\mathfrak{a}$ for $m,n \in \mathbb{Z}$. 
Then we have $m \equiv n ~(\bmod\,{p^v})$ with $\norm(\mathfrak{p})=p$. 
\end{lemma}

\begin{proof}[Proof of Theorem \ref{thm:1.2}]
Denote by $M_\alpha(x)$ the set of integers $n$ in $[1,x]$ such that there exists a prime ideal $\mathfrak{p}$ of $K$ with conditions \eqref{eq:10022256}. 
Furthermore, we also define $N_f(x)$ as the set of integers $n$ in $[1,x]$ such that $f(n)$ is not $x$-smooth. 
Then we have
\begin{align}\label{eq:10031220}
|M_\alpha(x)|=C_\alpha(x)
\quad\text{and}\quad
|N_f(x)|=x-\Psi_f(x,x)+O(1)
\end{align}
with an absolute implied constant. 
Take an integer $n \in N_f(x)$ arbitrarily. 
Then there exists a rational prime $p$ such that $p \mid f(n)$ and $p>x$. 
Here, we suppose $x>\max\{C(\alpha), \norm(\mathfrak{a}), |A_f|\}$ as in Lemma \ref{lem:5.2}. 
Then we deduce from Lemma \ref{lem:5.2} that there exists a prime ideal $\mathfrak{p} \in P_\alpha$ such that $\mathfrak{p} \mid (n+\alpha)\mathfrak{a}$ and $\norm(\mathfrak{p})>x$. 
If $\mathfrak{p} \mid (m+\alpha)\mathfrak{a}$ is satisfied for some $m \in \mathbb{Z}$ in $[1,x]$, then we apply Lemma \ref{lem:5.3} to deduce that $m \equiv n ~(\bmod\,{p})$ with $\norm(\mathfrak{p})=p$. 
Since $m,n \in [1,x]$ and $p>x$, we must obtain $m=n$, and therefore, the condition
\begin{align*}
\mathfrak{p} \nmid \prod_{\substack{1 \leq m \leq x \\ m \neq n}} (m+\alpha)\mathfrak{a}
\end{align*}
is satisfied. 
As a result, we find that $n \in M_\alpha(x)$ holds if $n \in N_f(x)$. 
Hence the inclusion $N_f(x) \subset M_\alpha(x)$ follows. 

Next, we define $\phi(n)$ for $n \in \mathbb{Z}$ as the set of all prime ideals $\mathfrak{p}$ of $K$ such that conditions \eqref{eq:10022256} are satisfied. 
Then it is non-empty for any $n \in M_\alpha(x)$ by definition.  
On the other hand, we have $\phi(m) \cap \phi(n)=\emptyset$ for any integers $m,n \in M_\alpha(x)$ with $m \neq n$ by conditions \eqref{eq:10022256}. 
Furthermore, if $n \in M_\alpha(x) \setminus N_f(x)$, then any prime ideal $\mathfrak{p} \in \phi(n)$ satisfies $\norm(\mathfrak{p}) \leq x$ by Lemma \ref{lem:5.2}. 
Hence, we see that
\begin{align*}
\left|M_\alpha(x) \setminus N_f(x)\right|
\leq \pi_K(x)
\ll \frac{x}{\log{x}}
\end{align*}
by the prime ideal theorem, where the implied constant depends only on $K$. 
It yields the desired result by formulas \eqref{eq:10031220}. 
\end{proof}

\begin{proof}[Proof of Corollary \ref{cor:1.3}]
Let $f(t) \in \mathbb{Z}[t]$ be an irreducible polynomial with $f(-\alpha)=0$. 
Then $f(t)$ is of degree $d=\deg(\alpha) \geq2$. 
By Theorem \ref{thm:1.1} with $u=1$, we deduce from Theorem \ref{thm:1.2} that
\begin{align*}
C_\alpha(x)
> x- \frac{\gamma_f(1)}{d} \left(x+O\left(\frac{x}{\log\log{x}}\right)\right) 
+ O\left(\frac{x}{\log{x}}\right), 
\end{align*}
where the implied constants depend only on $f(t)$ and $K$. 
It yields the inequality
\begin{align*}
C_\alpha(x)
\geq \left(1-\frac{\gamma_f(1)}{d}\right)x
\end{align*}
for sufficiently large $x$. 
Furthermore, the coefficient is calculated as
\begin{align*}
1-\frac{\gamma_f(1)}{d}
= 1-\frac{1}{2d}-\frac{3}{16d^2}-\frac{1}{d}\sqrt{\frac{3}{16d}+\left(\frac{3}{16d}\right)^2} 
\end{align*}
by using \eqref{eq:10031455}. 
We also obtain
\begin{align*}
1-\frac{\gamma_f(1)}{d}
> 1-\frac{0.914}{2}
= 0.543
\end{align*}
by \eqref{eq:10031502} and $d \geq2$. 
Hence the proof is completed. 
\end{proof}

Let $f(t)$ be an irreducible polynomial with $f(-\alpha)=1$. 
If \eqref{eq:09251118} is true at least for $u=1$, then we deduce from Theorem \ref{thm:1.2} that
\begin{align*}
C_\alpha(x)
\sim (1-\rho(d))x
\end{align*}
as $x \to\infty$, where we put $d=\deg(\alpha) \geq2$. 
Here, it is known that $\rho(d)=d^{-d+o(d)}$ as $d \to\infty$. 
Thus the result of Theorem \ref{thm:1.2} still leaves room for improvement. 

Remark that the formulation of Corollary \ref{cor:1.3} is due to Worley \cite{Worley1970}, and the original statement by Cassels in \cite{Cassels1961} is slightly different. 
To close this section, we confirm the original statement of Cassels. 

\begin{theorem}[Cassels \cite{Cassels1961}]\label{thm:5.6}
Let $\alpha$ be an algebraic irrational number. 
Then there exists an integer $N_0=N_0(\alpha)>10^6$ depending only on $\alpha$ with the following property. 
Suppose that $N \geq N_0$ and put $M = \lfloor 10^{-6} N \rfloor$. 
Then at least $51M/100$ integers in $N<n \leq N+M$ are such that $(n+\alpha) \mathfrak{a}$ is divisible by a prime ideal $\mathfrak{p}$ which does not divide $(m+\alpha) \mathfrak{a}$ for any integer $0 \leq m \leq N+M$ with $m \neq n$. 
\end{theorem}

\begin{proof}
The only difference is the intervals where $n$ belongs. 
To show Theorem \ref{thm:5.6}, we apply Propositions \ref{prop:4.3} and \ref{prop:4.4} with $x=N+M$, $z=M$, and $m=u=1$. 
Then we see that inequality \eqref{eq:09272331} holds with
\begin{align*}
V
&= \frac{M (\log(N+M)+O(1))}{2 \log{f(N)}}
+ O\left(\frac{N+M}{\log{N}}\right), \\
W
&= \frac{3M (\log(N+M)+O(1))^2}{8(\log{f(N)})^2}
+ O\left(\frac{N+M}{(\log{N})^2}\right) 
\end{align*}
for any polynomial $f(t)$ satisfying the assumptions of Proposition \ref{prop:4.3} and \ref{prop:4.4}. 
Thus we obtain
\begin{align*}
\Psi_f(N+M,N+M)-\Psi_f(N,N+M)
< \frac{\kappa_f(1)}{d}M 
< 0.48M
\end{align*}
by \eqref{eq:10022027} if $N$ is large enough. 
Furthermore, along the same line as the proof of Theorem \ref{thm:1.2}, we see that the number of integers in $N<n \leq N+M$ with the desired property is equal to
\begin{align*}
M - \left(\Psi_f(N+M,N+M)-\Psi_f(N,N+M)\right)
+ O\left(\frac{N+M}{\log(N+M)}\right), 
\end{align*}
where $f(t) \in \mathbb{Z}[t]$ is an irreducible polynomial with $f(-\alpha)=0$. 
Hence we obtain the conclusion. 
\end{proof}

\section{Proof of Theorem \ref{thm:1.4} and its corollaries}\label{sec:6}
Firstly, we recall the following criterions, which yield $N(x)=R_1(x)$ immediately. 

\begin{lemma}[Everest--Stevens--Tamsett--Ward \cite{EverestStevensTamsettWard2007}]\label{lem:6.1}
Let $n \in \mathbb{Z}_{>|b|}$ with $-b$ not an integer square. 
Then $n^2+b$ has a primitive divisor if and only if $P^+(n^2+b)>2n$. 
\end{lemma}

\begin{lemma}[Todd \cite{Todd1949}]\label{lem:6.2}
Let $n \in \mathbb{Z}_{\geq1}$. 
Then $\arctan n$ is irreducible if and only if $P^+(n^2+1)>2n$. 
\end{lemma}

Furthermore, Theorem \ref{thm:1.4} is proved by using Lemma \ref{lem:6.1} as follows. 

\begin{proof}[Proof of Theorem \ref{thm:1.4}]
By Lemma \ref{lem:6.1}, $R_b(x)$ is evaluated as
\begin{align*}
R_b(x)
= x - \sum_{\substack{1 \leq n \leq x \\ P^+(n^2+b)<2n}} 1 +O(1)
\end{align*}
for $x \geq1$, where the implied constant depends only on $|b|$. 
Furthermore, we obtain
\begin{align*}
\sum_{\substack{1 \leq n \leq x \\ P^+(n^2+b)<2n}} 1
&= \sum_{\substack{1 \leq n \leq x/2 \\ P^+(n^2+b)<2n}} 1
+ \sum_{\substack{x/2< n \leq x \\ P^+(n^2+b)<2n}} 1 \\ 
&= \sum_{\substack{1 \leq n \leq x/2 \\ P^+(n^2+b) \leq x}} 1
- \sum_{\substack{1 \leq n \leq x/2 \\ 2n \leq P^+(n^2+b) \leq x}} 1
+ \sum_{\substack{x/2< n \leq x \\ P^+(n^2+b) \leq x}} 1 
+ \sum_{\substack{x/2< n \leq x \\ x<P^+(n^2+b)<2n}} 1 \\
&= S_1-S_2+S_3+S_4, 
\end{align*}
say. 
Then $S_1+S_3=\Psi_f(x,x)$ holds with $f(t)=t^2+1$ by definition. 
On the other hand, $S_2$ and $S_4$ are evaluated as
\begin{align*}
S_2
&\leq \sum_{\substack{1 \leq n \leq z \\ 2n \leq P^+(n^2+b) \leq x}} 1
+ \sum_{\substack{z \leq n \leq x/2 \\ 2n \leq P^+(n^2+b) \leq x}} 1
\leq z + \sum_{\substack{z \leq n \leq x \\ 2z \leq P^+(n^2+b) \leq x}} 1, \\
S_4
&\leq \sum_{\substack{z< n \leq x \\ 2z<P^+(n^2+b)<2x}} 1
\end{align*}
with a parameter $z$ such that $1\leq z \leq x/2$. 
From the above, we obtain the formula
\begin{align}\label{eq:10111726}
R_b(x)
= x - \Psi_f(x,x) 
+ O\left(z+\sum_{\substack{z \leq n \leq x \\ 2z \leq P^+(n^2+b)<2x}} 1\right)
\end{align}
for $x \geq1$. 
Here, we see that
\begin{align*}
\sum_{\substack{z \leq n \leq x \\ 2z \leq P^+(n^2+b)<2x}} 1
= \sum_{2z \leq p < 2x} \sum_{\substack{z \leq n \leq x \\ P^+(n^2+b)=p}} 1
\leq \sum_{2z \leq p < 2x} 
\sum_{\substack{z \leq n \leq x \\ f(n) \equiv 0~ \hspace{-2mm} \pmod{p}}} 1
\end{align*}
holds since $P^+(n^2+b)=p$ implies $f(n) \equiv 0 ~(\bmod\,{p})$. 
We have $\omega_f(p) \leq 4 \sqrt{|b|}$ by applying \eqref{eq:09292100}. 
Therefore, the last sum is further evaluated as
\begin{align*}
\sum_{2z \leq p < 2x} 
\sum_{\substack{z \leq n \leq x \\ f(n) \equiv 0 \hspace{-2mm} \pmod{p}}} 1
&\ll \sum_{2z \leq p < 2x} \left(\frac{x}{p}+1\right) \\
&= x \log \left(\frac{\log(2x)}{\log(2z)}\right)
+ O\left(\frac{x}{\log(2z)}\right)
\end{align*}
by using the well-known formula
\begin{align*}
\sum_{p \leq t} \frac{1}{p}
= \log\log{t}+A+O\left(\frac{1}{\log{t}}\right). 
\end{align*}
Here, we choose the parameter $z$ as $z=x/(2\log{x})$. 
In this case, we have 
\begin{align*}
\log \left(\frac{\log(2x)}{\log(2z)}\right)
= \log \left\{\left(1+\frac{\log{2}}{\log{x}}\right)\left(1-\frac{\log\log{x}}{\log{x}}\right)^{-1}\right\}
\ll \frac{\log\log{x}}{\log{x}}
\end{align*}
for sufficiently large $x$. 
Hence we arrive at the estimate
\begin{align*}
\sum_{\substack{z \leq n \leq x \\ 2z \leq P^+(n^2+b)<2x}} 1
\ll \frac{x \log\log{x}}{\log{x}}, 
\end{align*}
which yields the desired result by \eqref{eq:10111726}. 
\end{proof}

\begin{proof}[Proof of Corollaries \ref{cor:1.5} and \ref{cor:1.6}]
The proof is similar to that of Corollary \ref{cor:1.3}. 
Let $f(t)=t^2+b$. 
By Theorem \ref{thm:1.1} with $u=1$, we deduce from Theorem \ref{thm:1.4} that
\begin{align*}
R_b(x)
> x- \frac{\gamma_f(1)}{2} \left(x+O\left(\frac{x}{\log\log{x}}\right)\right) 
+ O\left(\frac{x}{\log{x}}\right), 
\end{align*}
which implies the inequality
\begin{align*}
R_b(x)
\geq \left(1-\frac{\gamma_f(1)}{2}\right)x
\end{align*}
for sufficiently large $x$. 
Since we have 
\begin{align*}
1-\frac{\gamma_f(1)}{2}
> 1-\frac{0.914}{2}
= 0.543
\end{align*}
by using \eqref{eq:10031502}, we obtain the result for $R_b(x)$. 
The result for $N(x)$ also follows due to $N(x)=R_1(x)$ as explained before. 
\end{proof}

%\bibliographystyle{amsplain}
%\bibliography{refs}

\begin{thebibliography}{10}

\bibitem{Andersson2016+}
J.~Andersson, \emph{On questions of {C}assels and {D}rungilas-{D}ubickas},
  2016, preprint, \url{https://arxiv.org/abs/1606.02524}.

\bibitem{Cassels1961}
J.~W.~S. Cassels, \emph{Footnote to a note of {D}avenport and {H}eilbronn}, J.
  London Math. Soc. \textbf{36} (1961), 177--184. \MR{146359}

\bibitem{ChowlaTodd1949}
S.~D. Chowla and J.~Todd, \emph{The density of reducible integers}, Canad. J.
  Math. \textbf{1} (1949), 297--299. \MR{30558}

\bibitem{DavenportHeilbronn1936a}
H.~Davenport and H.~Heilbronn, \emph{On the {Z}eros of {C}ertain {D}irichlet
  {S}eries}, J. London Math. Soc. \textbf{11} (1936), no.~3, 181--185.
  \MR{1574345}

\bibitem{Dickman1930}
K.~Dickman, \emph{On the frequency of numbers containing prime factors of a
  certain relative magnitude}, Ark. Mat. Astron. Fys. \textbf{22} (1930),
  no.~10, 1--14.

\bibitem{EverestHarman2008}
G.~Everest and G.~Harman, \emph{On primitive divisors of {$n^2+b$}}, Number
  theory and polynomials, London Math. Soc. Lecture Note Ser., vol. 352,
  Cambridge Univ. Press, Cambridge, 2008, pp.~142--154. \MR{2428520}

\bibitem{EverestStevensTamsettWard2007}
G.~Everest, S.~Stevens, D.~Tamsett, and T.~Ward, \emph{Primes generated by
  recurrence sequences}, Amer. Math. Monthly \textbf{114} (2007), no.~5,
  417--431. \MR{2309982}

\bibitem{Granville2008}
A.~Granville, \emph{Smooth numbers: computational number theory and beyond},
  Algorithmic number theory: lattices, number fields, curves and cryptography,
  Math. Sci. Res. Inst. Publ., vol.~44, Cambridge Univ. Press, Cambridge, 2008,
  pp.~267--323. \MR{2467549}

\bibitem{Harman2024}
G.~Harman, \emph{Two problems on the greatest prime factor of {$n^2+1$}}, Acta
  Arith. \textbf{213} (2024), no.~3, 273--287. \MR{4749198}

\bibitem{HildebrandTenenbaum1993}
A.~Hildebrand and G.~Tenenbaum, \emph{Integers without large prime factors}, J.
  Th\'eor. Nombres Bordeaux \textbf{5} (1993), no.~2, 411--484. \MR{1265913}

\bibitem{Hmyrova1966}
N.~A. Hmyrova, \emph{On polynomials with small prime divisors. {II}}, Izv.
  Akad. Nauk SSSR Ser. Mat. \textbf{30} (1966), 1367--1372. \MR{207660}

\bibitem{Huxley1981}
M.~N. Huxley, \emph{A note on polynomial congruences}, Recent progress in
  analytic number theory, {V}ol. 1 ({D}urham, 1979), Academic Press, London-New
  York, 1981, pp.~193--196. \MR{637347}

\bibitem{Kowalski2004}
E.~Kowalski, \emph{On the ``reducibility'' of arctangents of integers}, Amer.
  Math. Monthly \textbf{111} (2004), no.~4, 351--354. \MR{2057190}

\bibitem{Li2024+}
R.~Li, \emph{On the primitive divisors of quadratic polynomials}, 2024,
  preprint, \url{https://arxiv.org/abs/2406.07575}.

\bibitem{Martin2002}
G.~Martin, \emph{An asymptotic formula for the number of smooth values of a
  polynomial}, J. Number Theory \textbf{93} (2002), no.~2, 108--182.
  \MR{1899301}

\bibitem{Mine2022+}
M.~Mine, \emph{A random variable related to the {H}urwitz zeta-function with
  algebraic parameter}, to appear in Acta Arith. 

\bibitem{Nagel1921}
T.~Nagel, \emph{G{\'e}n{\'e}ralisation d'un th{\'e}or{\`e}me de {T}chebycheff},
  Journ. de Math. (8) \textbf{4} (1921), 343--356.

\bibitem{Suzuki2025+}
Y.~Suzuki, \emph{On even-odd amicable pairs}, 2025, preprint.

\bibitem{Timofeev1977}
N.~M. Timofeev, \emph{Polynomials with small prime divisors}, Ta\v skent. Gos.
  Univ. Nau\v cn. Trudy (1977), no.~548, 87--91. \MR{565981}

\bibitem{Todd1949}
J.~Todd, \emph{A problem on arc tangent relations}, Amer. Math. Monthly
  \textbf{56} (1949), 517--528. \MR{31496}

\bibitem{Worley1970}
R.~T. Worley, \emph{On a result of {C}assels}, J. Austral. Math. Soc.
  \textbf{11} (1970), 191--194. \MR{0268156}

\end{thebibliography}

\providecommand{\bysame}{\leavevmode\hbox to3em{\hrulefill}\thinspace}
\providecommand{\MR}{\relax\ifhmode\unskip\space\fi MR }
\providecommand{\MRhref}[2]{%
  \href{http://www.ams.org/mathscinet-getitem?mr=#1}{#2}
}
\providecommand{\href}[2]{#2}

\end{document}